\documentclass[10pt, a4paper]{article}
\usepackage{amsmath, amsfonts}
\usepackage[usenames]{color}
\usepackage[utf8]{inputenc}
\usepackage{mathtools}
\usepackage{multirow}
\usepackage{hyperref} 
\hypersetup{bookmarks = true, 
                  colorlinks = true, 
                  linkcolor = blue, 
                  citecolor = blue, 
                  }
\usepackage{graphicx}

\usepackage{amssymb}
\usepackage{amsmath}

\newcommand{\C}{\mathbb{C}}

\newtheorem{theorem}{Theorem}[section]

\newtheorem{corollary}[theorem]{Corollary}

\newtheorem{definition}[theorem]{Definition}
\newtheorem{example}[theorem]{Example}

\newtheorem{lemma}[theorem]{Lemma}

\newtheorem{proposition}[theorem]{Proposition}
\newtheorem{remark}[theorem]{Remark}

\newenvironment{proof}[1][Proof]{\noindent\textbf{#1.} }{\ \rule{0.5em}{0.5em}}

\usepackage{xcolor}

\title{\textbf{On the fifth Whitney cone of a complex analytic curve}}
\author{ \ \ \ \\{A. Giles Flores, $\ $ $\ \ $ O.N. Silva $\ $ and $\ \ $ J. Snoussi}}
\date{}

\begin{document}

\maketitle

\begin{abstract}
From a procedure to calculate the  $C_5$-cone of a reduced complex analytic curve $X \subset \mathbb{C}^n$ at a singular point $0 \in X$, we extract a collection of integers that we call {\it auxiliary multiplicities} and we prove they characterize the Lipschitz type of complex curve singularities.  We then use them to improve the known bounds for the number of irreducible components of the $C_5$-cone.  We finish by giving 
an example showing that in a Lipschitz equisingular family of curves the  
number of planes in the $C_5$-cone may not be constant.
\end{abstract}

\section{Introduction}\label{introduction}

$ \ \ \ \ $ In \cite[Sec. 3]{whitney}, Whitney introduced some spaces that are now known in the literature as Whitney cones. Given an analytic set $W$ in $\mathbb{C}^n$ and a point $p\in W$, Whitney defined six types of cones, $C_1(W,p), C_2(W,p),\cdots,C_6(W,p)$, all of them having the point $p$ as vertex. In this work we will deal with the cones $C_3(W,p)$, $C_4(W,p)$ and $C_5(W,p)$ and for simplicity we will suppose that $p$ is the origin. Roughly speaking, the cone $C_3(W,0)$, known as the Zariski tangent cone, is the set of limit positions of secants of $W$ passing through $0$. The cone $C_4(W,0)$ is the set of all limits of tangent vectors to $X$ at $0$. Finally the cone $C_5(W,0)$ is the set of all limit positions of bi-secants, {\it i.e.}, limits of lines passing through a couple of points both converging to $0$; see Section \ref{sec2}, and also \cite[p. 92]{chirka} for precise definitions.

Whitney cones have proven to be very useful in singularity theory. For instance, in \cite[Sec. 1 and 2]{stutz} Stutz gives conditions for the weak (respectively, strong) equisingularity of a family $(W,0)$ of germs of complex analytic sets in terms of the dimensions of the cones $C_4(W,0)$ and $C_5(W,0)$. Also, if $(X,0)$ is a germ of singular reduced curve, then the cone $C_5(X,0)$ determines the set of all projections of $(X,0)$ to $(\mathbb{C}^2,0)$ such that the image is a curve with minimal Milnor number (see \cite[Prop. IV.2]{briancon}). 

It is well known that if $I \subset \mathbb{C}\{X_1, \ldots , X_n\}$ is a defining ideal of an analytic germ 
$(W,0)\subset (\mathbb{C}^n,0)$, then by using standard bases, one can find generators $f_1, \ldots , f_k$ of the ideal $I$ 
such that the ideal generated by the initial forms $In_0(f_i), i=1, \ldots , k$, is a defining ideal of the cone $C_3(W,0)$. 
In particular this gives the $C_3$-cone an algebraic structure. Whitney also provided the cones $C_4$ and $C_5$ with 
an algebraic structure; see \cite[Th. 5.6]{whitney}. However there is no known canonical method to determine defining
equations for these cones in all dimensions. 

In the case where $(X,0)$ is a germ of curve singularity, limits of secants and limits of tangents coincide. Consequently the sets 
$C_3(X,0)$ and $C_4(X,0)$ are the same, and these are a finite union of lines that can be determined directly either 
from the equations of $(X,0)$ or from the parametrizations of each of its branches.  

In 1972, Stutz showed that if $(X,0)$ is an analytic singular curve in $(\mathbb{C}^n,0)$, then $C_5(X,0)$ has dimension $2$ 
(\cite[Lem. 3.15]{stutz}). Some years later, Brian\c{c}on, Galligo and Granger showed that if $(X,0)$ is a singular 
germ of reduced curve, then $C_5(X,0)$ is a finite union of planes, each of them containing at least one 
tangent of $(X,0)$ (\cite[Prop. IV.1]{briancon}). 

In 2002, Krasi\'nski, gave in \cite{Krasinsky} formulas to determine the planes of the 
$C_5$-cone of a curve, starting from the parametrizations of each of 
its branches. He also gave a bound on the number of these planes.

In this work, we start by describing the procedure given by 
Krazi\'nski in \cite{Krasinsky}. We call it the ``$C_5$-procedure'' and present it in  Theorem \ref{generalcase1}.

The main idea is the following: For each branch $(X^{(i)},0)$ of 
$(X,0)$ of multiplicity $m^{(i)}$ and for each pair of branches $(X^{(j)},0)$ 
and $(X^{(k)}, 0)$ of $(X,0)$ of multiplicities $m^{(j)}$ and $m^{(k)}$,  
we construct a collection of analytic maps 
depending on the $m^{(i)}$-th roots of unity and the Puiseux
parametrization of $(X^{(i)},0)$ on one hand, and on the $m^{(k)}$-th roots of unity and the parametrizations of $(X^{(j)},0)$ and $(X^{(k)},0)$ on the other hand suposing
that $m^{(k)} \leq m^{(j)}$.  We call 
these maps {\it auxiliary parametrizations} associated to the curve $(X,0)$.
The images of these maps are a collection of curves
that we call {\it auxiliary curves} associated to $(X,0)$. When $(X,0)$ is a singular curve, the tangent lines of the auxiliary curves, along with the tangent lines to the branches of $(X,0)$ give rise to a finite union of planes which form the cone $C_5(X,0)$.  
We give precise definitions of the auxiliary curves and parametrizations in Section \ref{preliminaries}.

This strategy was already at the heart of the proof, by 
Brian\c{c}on, Galligo and Granger in \cite[Prop. IV.1]{briancon}, 
that the $C_5$-cone is a finite union of planes.  
Since it doesn't seem to be widespread knowledge among the community
we chose to present an explicit procedure giving a precise definition of the planes of the $C_5$-cone of a curve.  This is the aim of the first part of this work and it is explained in Sections \ref{preliminaries} and \ref{c5cone}.

In a second step, we exhibit from the auxiliary parametrizations a collection of integers that we call auxiliary multiplicities. We show that they determine the bi-Lipschitz type of the curve singularity. 

These numbers were first used by Pham and Teissier in \cite{pham} 
to study the saturation of local analytic algebras of dimension one.
They show in particular that for plane curve singularities, the auxiliary multiplicities determine the topological type of the curve. We recall these results in the first part of Section \ref{sec4}. 
  
Neumann and Pichon also used the auxiliary multiplicities to prove 
a result by Teissier stating that,  for a complex curve germ $(X,0) 
\subset (\mathbb{C}^n,0)$, the restriction to $X$ of a generic linear projection to $\mathbb{C}^2$ is bi-Lipschitz for the outer geometry \cite[Th. 5.1]{pichon}. In their work, they call these numbers \textit{``essential integer exponents''}.

Also in Section \ref{sec4}, we show that for a reduced curve singularity $(X,0)$,  
the auxiliary multiplicities are bi-Lipschitz invariants of $(X,0)$ for the 
outer metric.  More precisely, we show that two germs of curves $(X,0)$ and $(Y,0)$ in $(\mathbb{C}^n,0)$ are bi-Lipschitz equivalent if and only if there exists a bijection between its branches preserving all the auxiliary multiplicities (Theorem \ref{bilipequivalence}). 
 
 In Section \ref{bounds}, we address the question of the number of planes that appear in the $C_5$-cone of a curve singularity $(X,0)$. Using our previous results
 we are able to improve the known upper bounds for the number of irreducible components of $C_5(X,0)$  in  propositions \ref{bound2} and \ref{bound3}.

In the case of the $C_3$-cone of curves, it is known that the number of 
irreducible components of the cone need not be constant in Whitney equisingular families of curves, see \cite[Ex. 4.13]{otoniel6}. It has also been proved that in bi-Lipschitz equisingular 
families of curves, the number of irreducible components of the $C_3$-cone is constant, see \cite{otoniel6} and also \cite{edson} for a more general situation. Thus, we have a natural question: Is the number of planes of $C_5(X,0)$ a bi-Lipschitz invariant for a curve $(X,0)$? In Section \ref{number}, we give a negative answer to this question
with an example of a Lipschitz regular family of curves where the number 
of planes in the $C_5$-cone is not constant. We finish by using this to construct curves which are bi-Lipschitz equivalent but not analytically equivalent.\\ 
 
Before starting, we will establish some notations that will be used throughout this work.\\

$\bullet$ Unless stated otherwise, $W$ denotes an arbitrary analytic set in $\mathbb{C}^n$ and we will denote the non-singular locus of $(W,0)$ by $reg(W,0)$. On the other hand, $(X,0)$ and $(Y,0)$ will denote germs of reduced analytic (singular or smooth) curves in $(\mathbb{C}^n,0)$, with $n\geq 2$.

$\bullet$ When $(X,0)$ and $(Y,0)$ are germs of plane curves, the number $i(X,Y)$ denotes the intersection multiplicity of $(X,0)$ and $(Y,0)$.

$\bullet$ $(X^{(i)},0)$ denotes an irreducible component, or a branch of a curve $(X,0)$. A parametrization of $(X^{(i)},0)$ will be denoted by $\varphi^{(i)}=(\varphi^{(i)}_1,\cdots, \varphi^{(i)}_n)$.

$\bullet$ $m(X,0)$ or $m$ (for short) denotes the multiplicity of $(X,0)$ and $m^{(i)}$ denotes the multiplicity of $(X^{(i)},0)$.
 
$\bullet$ $G_m$ denotes the cyclic group with $m$ elements ($m$-roots of unity). If $\theta \in G_m$, then $ord(\theta)$ denotes the order of $\theta$ in $G_m$. 

$\bullet$ Given $\varphi(u)=\sum a_i u^i$ a non-zero element in $\mathbb{C}\lbrace u \rbrace$, the ring of convergent power series in the variable $u$, $ord_0(\varphi)$ denotes the order of $\varphi$ at $0$, i.e., the minimum integer $n$ such that $a_n \neq 0$. Set $ord_0(0)=\infty$.

\section{Preliminaries}\label{preliminaries}

\subsection{The Whitney cones $C_3$, $C_4$ and $C_5$}\label{sec2}

$ \ \ \ \ $ Let $W$ be an analytic set in $\mathbb{C}^n$, and $p \in W$. As stated in the introduction, for simplicity we will assume that the point $p$ is the origin. Following \cite[p. 91]{chirka}, we start by defining the cones $C_3(W,0)$, $C_4(W,0)$ and $C_5(W,0)$; throughout this work we will only deal with these three cones. 

\begin{definition} Let $v$ be a vector in $\mathbb{C}^n$.\\

\noindent (a) We say that $v \in C_3(W,0)$ if there exist a sequence of points $(w_s) \in W$ 
and a sequence of complex numbers $(\lambda_s)$ such that $(w_s)\rightarrow 0$ and $(\lambda_s w_s)\rightarrow v$ as $s\rightarrow \infty$.\\

\noindent(b) We say that $v \in C_4(W,0)$ if there are sequences of points $(w_s) \in reg(W)$ and vectors $(v_s) \in T_{w_s}W$ such that $(w_s) \rightarrow 0$ and $(v_s)\longrightarrow v$ as $s\longrightarrow \infty$.\\

\noindent(c) We say that $v \in C_5(W,0)$ if there are distinct sequences of points $(w_s),(w_s') \in W$ and numbers 
$(\lambda_s) \in \mathbb{C}$ such that $(w_s)\rightarrow 0$, $(w_s')\rightarrow 0$ 
and $\lambda_s(w_s-w_s')\rightarrow v$ as $s\rightarrow \infty$. 
\end{definition}

\begin{remark}\label{coneremark} \noindent (a) The cone $C_3(W,0)$ is made of lines through $0$ obtained as limits of secant lines with direction $\overline{w_s-0}$ where $(w_s)\subset W\setminus \{0\}$ is a sequence of points converging to $0$. The cone $C_4(W,0)$ is the union of all limits of tangent spaces $T_{w_s}W$ to $W$ at non-singular points $w_s$ converging to $0$. 
The cone $C_5(W,0)$ is the set of all the lines obtained as limits of secant lines with direction $\overline{w_s - w'_s}$ 
through distinct sequence of points $w_s$ and $w'_s$ of $W$, both converging to $0$.\\ 

\noindent (b) Whitney provided the cones $C_3, C_4$ and $C_5$ with an algebraic  structure (see \rm\cite[\textit{Sec}. 5]{whitney})\textit{, making them possibly non-reduced algebraic spaces. However, in this work we are only interested in their reduced structure, or their set theoretic construction.}\\ 

\noindent \textit{(c) The cone $C_3(W,0)$ is sometimes called Zariski's tangent cone. It is a well described and widely studied space, see for example} \rm\cite[\textit{Ch.} 2]{chirka}. \textit{In the case of a germ of curve $(X,0)$ in $(\mathbb{C}^n,0)$, 
we know that limits of secants through $0$ and limits of tangents of $(X,0)$ at $0$ coincide; see for example} \rm\cite[\textit{Prop.} 2.3.4]{jsnoussi-LNM}. \textit{Hence, the cones $C_3(X,0)$ and $C_4(X,0)$, both provided with the reduced structure, are the same space. Therefore, the next natural step is to study the $C_5$-cone of $X$.}\\

\noindent \textit{(d) We have that $C_3(W,0)\subset C_4(W,0) \subset C_5(W,0)$ and $C_i(W^{(1)} 
\cup W^{(2)},0) = C_i(W^{(1)},0) \cup C_i(W^{(2)},0)$ for $i=3,4$. However, in general}

\begin{center}
 $  C_5(W^{(1)},0) \cup C_5(W^{(2)},0) \subseteq C_5(W^{(1)} \cup W^{(2)},0)$,
 \end{center} 

\noindent \textit{and this inclusion can be proper, as we will see in Example} \rm\ref{exe2}.

\end{remark}

The following Theorem is due to Brian\c{c}on, Galligo and Granger.

\begin{theorem}\label{BGG}\rm(\cite{briancon}, \textit{Th. $IV.1$}\rm) \textit{If $X$ is an analytic singular reduced complex curve, then $C_5(X,0)$ is a finite union of planes, each of them containing at least one tangent to $X$.}
\end{theorem}

We would like to describe more precisely how to find the planes of the cone $C_5(X,0)$. Inspired by the proof of Theorem \ref{BGG} in \cite[IV.1]{briancon}, and
using the results of \cite{Krasinsky} we will describe a procedure to build the $C_5$-cone of a curve. For that purpose, in the next section, we introduce the notion of {\it auxiliary curves} associated to $(X,0)$ that will play a fundamental role in what follows. 

\subsection{Auxiliary parametrizations, curves and multiplicities}\label{seccurvasaux}

$ \ \ \ \ $ Starting from the local parametrizations of the branches of a germ of complex analytic curve $(X,0)$, 
or equivalently, from its normalization, we will produce a collection of curves whose tangent lines allow us to 
give a precise description of the cone $C_5(X,0)$. 
Furthermore, multiplicities associated to these curves will happen to be bi-Lipschitz invariants of the original curve.

\begin{definition}
(a) Let $(X,0) \subset (\mathbb{C}^n,0)$ be a germ of irreducible 
and reduced curve. 
A \textit{parametrization} of $(X,0)$ is a finite holomorphic map germ:

\begin{center}
$\varphi:(\mathbb{C},0)\rightarrow (\mathbb{C}^n,0)$, $ \ \ \ $ $u\mapsto (\varphi_1(u),\cdots,\varphi_n(u))$
\end{center}

\noindent with $\varphi(\mathbb{C},0) = (X,0)$. 

If in addition, $\varphi$ satisfies the following factorization property: ``Each finite holomorphic map germ $\psi:(\mathbb{C},0)\rightarrow (\mathbb{C}^n,0)$, such that $\psi(\mathbb{C},0) = (X,0)$, factors in a unique way through $\varphi$, that is, there exists a unique holomorphic map germ $\widehat{\psi}:(\mathbb{C},0)\rightarrow (\mathbb{C},0)$ such that $\psi= \varphi \circ \widehat{\psi}$'', then we say that $\varphi$ is a \textit{primitive parametrization} of $(X,0)$. 

If $(X,0)=(X^{(1)}\cup \cdots \cup X^{(r)},0)$, then a parametrization of $(X,0)$ is a system of parametrizations $\lbrace \varphi^{(1)},\cdots, \varphi^{(r)} \rbrace$ of the branches $(X^{(i)},0)$.\\

\noindent (b) Let $(X,0)$ be a germ of irreducible curve in $(\mathbb{C}^n,0)$ with multiplicity $m$. 
We say that a primitive parametrization $\varphi: (\mathbb{C},0)\rightarrow (X,0)$ is a Puiseux parametrization of $(X,0)$ 
if $\varphi$ has the following form:

\begin{center}
 $\varphi(u)=(\varphi_1(u),\cdots,\varphi_n(u))= \left(\displaystyle { \sum_{i\geq m}^{}}a_{1,i}u^i, \   \cdots, \  \displaystyle { \sum_{i\geq m}^{}}a_{j-1,i}u^i, \ u^m \ , \displaystyle { \sum_{i\geq m}^{}}a_{j+1,i}u^i, \ \cdots \ , \ \displaystyle { \sum_{i\geq m}^{}}a_{n,i}u^i \right).$
 \end{center} 

\noindent In this case, the $j$-th coordinate is called a {\it special coordinate} for the parametrization $\varphi$. If $(X,0)=(X^{(1)}\cup \cdots \cup X^{(r)},0)$, then a Puiseux parametrization of $(X,0)$ is 
a system of Puiseux parametrizations $\lbrace \varphi^{(1)},\cdots, \varphi^{(r)} \rbrace$ of the branches $(X^{(i)},0)$.

We say that a system of Puiseux parametrizations $\lbrace \varphi^{(1)},\cdots, \varphi^{(r)} \rbrace$ is compatible 
if for any pair $(i,j)$ such that $X^{(i)}$ is tangent to $X^{(j)}$ the respective primitive Puiseux parametrizations of $X^{(i)}$ and $X^{(j)}$ have a common special coordinate (not necessarily with the same power ``$m$'').\\

\noindent (c) Suppose that $(X,0)=(X^{(1)}\cup \cdots \cup X^{(r)},0)$ and define the following sets:

\begin{flushleft}
$-$ $S(X)$ is the set of indices of singular branches, {\it i.e.,} the subset of $\lbrace 1,\cdots, r \rbrace$ defined by: 
$i \in S(X)$ if and only if $(X^{(i)},0)$ is singular.

$-$ $T(X)$ is the set of pairs of indices of tangent branches, {\it i.e.,} the set of pairs $(i,j)$ with $i<j$ such that 
$(X^{(i)},0)$ is tangent to $(X^{(j)},0)$.

$-$ $NT(X)$ is the set of pairs of indices of non-tangent branches, {\it i.e.}, pairs $(i,j)$ with $i<j$ such that 
$(X^{(i)},0)$ is not tangent to $(X^{(j)},0)$.
\end{flushleft}

\end{definition}

\begin{remark}\label{remarkstandard} It is well known that if $(X,0)$ is irreducible and  $\varphi$ is an arbitrary primitive
parametrization of $(X,0)$, then there is an analytic isomorphism $\xi:(\mathbb{C},0)\rightarrow (\mathbb{C},0)$ such that 
$\varphi  \circ \xi$ is a Puiseux parametrization of $(X,0)$. Therefore, given a curve $(X,0)$, we can always choose a 
Puiseux parametrization for each branch of $(X,0)$ (see for instance \rm\cite[\textit{p.} \rm 98]{chirka}). \textit{Furthermore, it is not hard to see that a compatible system of Puiseux parametrizations for a curve always exists. 
Consider a germ of curve $(X,0)=(X^{(1)} \cup X^{(2)},0)$ in $(\mathbb{C}^n,0)$, with two tangent irreducible components. 
Consider a system of Puiseux parametrizations $\lbrace \varphi^{(1)},\varphi^{(2)}\rbrace$ for $(X,0)$, defined as}

\begin{center}
$\varphi^{(1)}(u)= (\varphi^{(1)}_1, \ldots , \varphi^{(1)}_n)$ $ \ \ \ $ and $ \ \ \ $ $\varphi^{(2)} = (\varphi^{(2)}_1, \ldots , \varphi^{(2)}_n)$.
 \end{center} 

\textit{The system of parametrizations of $X$ is compatible if and only if there exists $j\in \{ 1, \ldots n\}$ such that}

\begin{center} 
$\varphi ^{(1)}_j (u) = u^{m^{(1)}}$ $ \ \ \ $ and $ \ \ \ $  $\varphi^{(2)}_j(u) = u^{m^{(2)}}$,
\end{center} 

\noindent \textit{where the integers $m^{(1)}$ and $m^{(2)}$ are respectively the lowest orders among the $\varphi^{(1)}_i$'s and $\varphi^{(2)}_i$'s, respectively. }

\textit{Such an index $j$ is a common special coordinate for $\varphi^{(1)}$ and $\varphi^{(2)}$.}

\end{remark}

We are now able to present our main definition. Given a finite analytic map germ $f:(\mathbb{C},0)\rightarrow (\mathbb{C}^n,0)$, it is well-known that the image of $f$ is 
an analytic set. However, we can consider different analytic structures on the image of $f$ 
(see for instance \rm\cite[p. 48]{greuel6}). In the next definition, we will define some curves as images of finite maps 
and we will adopt the reduced structure for these images. 

\begin{definition}(Auxiliary curves and multiplicities)\label{auxdef}
Let $(X,0)=(X^{(1)} \cup \cdots \cup X^{(r)},0)$ be a germ of reduced curve in $(\mathbb{C}^n,0)$.\\

\noindent \textbf{(a)} Suppose that $r=1$ so that $(X,0)=(X^{(1)},0)$ is irreducible and singular. 
Define $m^{(1)}:=m(X^{(1)},0)$ to be its multiplicity. 
Let $\varphi^{(1)}=(\varphi^{(1)}_1 ,\cdots,\varphi^{(1)}_n)$ be a 
Puiseux parametrization  of $(X^{(1)},0)$. For each $m^{(1)}$-th root of unity $\theta \neq 1 \in G_{m^{(1)}}$:\\

{(a.1)} Define the map:
$$\begin{array}{rcl}
\phi^{(1)}_{\theta} :(\mathbb{C},0) & \rightarrow & (\mathbb{C}^n,0)\\
u & \mapsto &\varphi^{(1)}(u)-\varphi^{(1)}(\theta u),
\end{array}$$
we will call it an auxiliary characteristic parametrization associated to 
$(X^{(1)},0)$. The image of $\phi^{(1)}_{\theta}$ will be called
an auxiliary characteristic curve associated to $(X^{(1)},0)$ and denoted 
by $(A^{(1)}_{\theta}(X),0)$.

{(a.2)} Define the $\mathbb{C}$-vector space $H^{(1)}_{\theta}$ to be the linear space generated by non-zero vectors in $C_3(X,0)$ 
and $C_3(A^{(1)}_{\theta}(X),0)$.

{(a.3)} Define $m^{(1)}_{\theta}:={\displaystyle min_j} \lbrace \ ord_0 [\varphi^{(1)}_j(u)-\varphi^{(1)}_j(\theta u)] \ \rbrace $. The collection 
$\lbrace m^{(1)} \rbrace \cup \lbrace m^{(1)}_{\theta} \ | \ \theta \in G_{m^{(1)}}\}$ 
is called the characteristic auxiliary multiplicities associated to $(X^{(1)},0)$. We will 
denote it by $ChAM(X^{(1)},0)$. \\

\noindent \textbf{(b)} Suppose that $(X,0)=(X^{(1)} \cup X^{(2)},0)$, set $m^{(i)}=m(X^{(i)},0)$ to be their respective multiplicities and 
suppose that $m^{(1)} \geq m^{(2)}$.

\noindent Choose compatible Puiseux parametrizations $\varphi^{(1)}$ and 
$\varphi^{(2)}$ for $(X^{(1)},0)$ and $(X^{(2)},0)$, respectively. Then, for each $\theta \in G_{m^{(2)}}$:\\

(b.1) Define the map:
$$\begin{array}{rcl}
\phi^{(1,2)}_{\theta} :(\mathbb{C},0) & \rightarrow & (\mathbb{C}^n,0)\\
u & \mapsto & \varphi^{(1)}( u^{ m^{(2)}})- \varphi^{(2)} ( \theta u^{m^{(1)}}),
\end{array}$$
and call it a contact auxiliary parametrization associated to $(X^{(1)},0)$ and $(X^{(2)},0)$. The image of the map 
$\phi^{(1,2)}_{\theta}$ will be called a contact auxiliary curve associated to $(X^{(1)},0)$ and $(X^{(2)},0)$, and denoted
by $(A^{(1,2)}_{\theta}(X),0)$.

(b.2) Let $w_1,w_2,w_{(1,2)}$ be non-zero vectors in $C_3(X^{(1)},0)$, $C_3(X^{(2)},0)$ and $C_3(A_{\theta}^{(1,2)}(X),0)$, respectively. Define the $\mathbb{C}$-vector space $H_{\theta}^{(1,2)}$ to be the linear space generated by $w_1,w_2,w_3$.

(b.3) Define $m^{(1,2)}_{\theta}:={\displaystyle min_j} \lbrace \ ord_0 [ \varphi^{(1)}_j( u^{ m^{(2)}})- \varphi^{(2)}_j( \theta u^{m^{(1)}})] \ \rbrace $. 
The ordered sequence of integers, with possible repetitions,  
$(m^{(1,2)}_{\theta_1}\leq \ldots \leq m^{(1,2)}_{\theta_{m^{(2)}}}, \theta_i\in G_{m^{(2)}})$ 
is called the contact auxiliary multiplicities 
of the pair $((X^{(1)},0),(X^{(2)},0))$. It will be denoted by $CoAM(X^{(1)}, X^{(2)},0)$.\\

\noindent \textbf{(c)} In general, suppose that $(X,0)=(X^{(1)}\cup \cdots \cup X^{(r)},0)$. Choose a compatible system of Puiseux parametrizations $\lbrace \varphi^{(1)},\cdots, \varphi^{(r)} \rbrace$ of $(X,0)$, set $m^{(i)}=m(X^{(i)},0)$ and suppose  $m^{(1)} \geq m^{(2)} \geq \cdots \geq m^{(r)}$. 

An auxiliary multiplicity of $(X,0)$ is either a characteristic auxiliary multiplicity associated to a branch, denoted by $m^{(i)}_{\theta}$, or a contact 
auxiliary multiplicity associated to a pair of branches, denoted by $m^{(i,j)}_{\theta}$.

In a similar way, an auxiliary curve of $(X,0)$ is a branch which is either a characteristic auxiliary curve associated to a branch of 
$(X,0)$, denoted by $A_{\theta}^{(i)}(X,0)$, or a contact auxiliary  curve associated to a pair of branches of $(X,0)$, denoted by $A_{\theta}^{(i,j)}(X,0)$. For each $(X^{(i)},0)$ (respectively, for each pair $(X^{(i)} \cup X^{(j)},0)$, with $i<j)$ we define the spaces $H^{(i)}_{\theta}$ (respectively, $H_{\theta}^{(i,j)}$) as in (a) and (b).

\end{definition}

\begin{remark}\label{remarkpham} (a) The terminology we use in this definition will be justified by Lemmas \rm\ref{lemma1},
 \ref{lemma2}, \ref{lemma3}, \textit{Remark} \rm\ref{remarkstandard} \textit{(b) and Proposition} \rm\ref{propgeneric}. \textit{Clearly, the notion of auxiliary parametrizations and curves depend on the choice of the compatible system of 
Puiseux parametrizations $\lbrace \varphi^{(1)},\cdots, \varphi^{(r)} \rbrace$. 
However, we will show that the auxiliary multiplicities do not depend on that choice; see Remarks} 
\rm\ref{gooddef1} \textit{and} \rm\ref{gooddef2}.  \textit{For this reason, in the sequel, we frequently omit mentioning the chosen compatible system of Puiseux parametrizations.}\\

\noindent \textit{(b) As we said in the introduction, the auxiliary multiplicities were first used by Pham and Teissier in} \rm\cite{pham} \textit{where they studied the saturation of local rings of analytic reduced curves.  
In that context, the auxiliary multiplicities were used to prove that the saturation of the local ring of a germ of plane curve $(X,0)$ determines, and 
is determined by the characteristic exponents and the intersection multiplicities of the branches of $(X,0)$} 
\rm (\textit{see} \rm\cite[\textit{Prop.} \rm VI.3.2]{pham}).\\

\noindent \textit{(c) We will see in Section} \rm\ref{c5cone} \textit{that the spaces $H^{(i)}_{\theta}$ and $H^{(i,j)}_{\theta}$ are in fact two dimensional $\mathbb{C}-$vector spaces (planes) in $\mathbb{C}^n$ and they are the irreducible components of $C_5(X,0)$. If $(i,j) \in T(X)$, then $C_3(X^{(i)},0)$ and $C_3(X^{(j)},0)$ are the same (as reduced complex spaces). Therefore, in this case $H_{\theta}^{(i,j)}$ can be generated by $w_i$ and $w_{(i,j)}$ or by $w_j$ and $w_{(i,j)}$. On the other hand, if $(i,j) \in NT(X)$, we will see in the proof of Theorem} \rm\ref{generalcase1}, following \cite[Prop. 2.6]{Krasinsky},  \textit{that $H_{\theta}^{(i,j)}$ is generated by $w_i$ and $w_j$, therefore it does not depend on $\theta$. In this case, $H_{\theta}^{(i,j)}$ will be denoted only by $H^{(i,j)}$. Furthermore, in this case, one can check that the contact auxiliary multiplicites 
$m^{(i,j)}_{\theta}$ do not depend on $\theta \in G_{m^{(2)}}$ and are all equal to $m^{(1)}m^{(2)}$.}

\end{remark}

\begin{example}\label{exe3} Consider the germ of curve $(X,0)=V(y^2-x^3,z^2-x^2y)$ in $(\mathbb{C}^3,0)$, where $(x,y,z)$ denotes a local 
system of coordinates in $\mathbb{C}^3$. We have that $\varphi(u)=(u^4,u^6,u^7)$ is a Puiseux parametrization of $(X,0)$ and $m(X,0)=4$. Table \rm\ref{tabela3} \textit{shows the auxiliary characteristic parametrizations and multiplicities of $(X,0)$, where $V(f)$ denotes the zero set of $f \in \mathcal{O}_3$}.

\begin{table}[!h]
\caption{Auxiliary characteristic parametrizations and multiplicities of $(X,0)$ of Example \ref{exe3} }\label{tabela3}
\centering
{\def\arraystretch{2}\tabcolsep=15pt 
\begin{tabular}{@{} c | l | c | c @{}}

\hline
			
 $ \ \ \ \ \ \ $ $\theta$ $ \ \ \ \ \ $ & Auxiliary characteristic parametrization $\phi_{\theta}$  & $H_{\theta}$ &  $ \ \ \ \ \ \ $ $m_{\theta}$  $ \ \ \ \ \ \ $ \\
  
\hline  
  $-1$ & $(0,0,2u^7)$ & $V(y)$ & $7$ \\

  $i$ & $(0,2u^6,(1+i)u^7)$ & $V(z)$ & $6$ \\
   
  $-i$ & $(0,2u^6,(1-i)u^7)$ & $V(z)$ & $6$ \\

\hline
\end{tabular}
}
\end{table}

\end{example}

\section{How to find the $C_5$-cone?}\label{c5cone}

$ \ \ \ \ $ In this section, inspired by the result of Brian\c{c}on Galligo and Granger (Theorem \ref{BGG}), and using the results of \cite{Krasinsky} we present a procedure 
to find the $C_5$-cone of a germ of curve $(X,0)$. We call it the ``$C_5$-procedure". 
The main idea is that in the set of all (infinite) limits of bi-secants, it is sufficient to find a finite number of 
directions in order to determine all planes of $C_5(X,0)$. 

Let us first notice that for a smooth branch, the $C_5$-cone is a line. In fact, in \cite[p. 92]{chirka}, Chirka 
explains that all Whitney cones coincide with the tangent space when they are considered at non-singular points. 

\begin{lemma}\label{smoothcase} Let $(X,0)$ be a germ of smooth curve in $(\mathbb{C}^n,0)$ and let $X$ be a representative of $(X,0)$. Then as reduced complex spaces, we have that:

$$C_5(X,0)=C_3(X,0).$$
\end{lemma}

\begin{proof} This is a direct consequence of the existence of a Puiseux parametrization. In fact, consider a parametrization
$$\varphi : u \mapsto (u, \varphi_2(u), \ldots, \varphi_n(u))$$
of the smooth curve, where the holomorphic functions $\varphi_i$ have order at $0$ at least two. 

Then consider two different sequences of points $\varphi(u_s)$ and $\varphi(v_s)$ with $(u_s)$ and $(v_s)$ both converging to $0$ as $s\rightarrow \infty$.
The line $\overline{\varphi(u_s)-\varphi(v_s)}$ is represented by the projective point 
$$(u_s - v_s: \varphi_2(u_s) - \varphi_2(v_s): \ldots : \varphi_n(u_s) -\varphi_n(v_s)).$$
Since for all $i=2,\cdots,n$ we have that $ord_0(\varphi_i) >1$, all the terms are then multiple of $u_s - v_s$, so that the limit line is represented by the projective point $(1:0 \ldots :0)$ which is the line of the cone $C_3(X,0)$. 
\end{proof}

\begin{theorem}\label{generalcase1}\textbf{($C_5$-procedure)} Let $(X,0)=(X^{(1)}\cup \cdots \cup X^{(r)},0)$ be a germ of singular
 curve in $(\mathbb{C}^n,0)$. Choose a parametrization for each $(X^{(i)},0)$ such that the system of parametrizations is 
 compatible. With the same notations as in Definition \rm\ref{auxdef} and 
 Remark \ref{remarkpham} $(c)$\textit{, consider the following procedure:}\\

\textbf{Step $1$}: \textit{For each $i \in S(X)$ and $\theta \in G_{m^{(i)}} \setminus \lbrace 1 \rbrace $ find $H^{i}_{\theta}$.}

\textbf{Step $2$}: \textit{For each $(i,j) \in T(X)$, with $i < j$ and assuming $m^{(i)} \geq m^{(j)}$, and $\theta \in G_{m^{(j)}}$ find $H^{(i,j)}_{\theta}$.}

\textbf{Step $3$}: \textit{For each $(i,j) \in NT(X)$ find $H^{(i,j)}$.}\\

\noindent Then,
\begin{center}
$C_5(X,0) = \left( \displaystyle { \bigcup_{ \begin{array}{c}
 \mbox{\scriptsize $ i\in S(X),$ }  \\
\mbox{\scriptsize $ \theta \in G_{m^{(i)}} \setminus \lbrace 1 \rbrace$ }     \end{array}   }^{}} H^{i}_{\theta} \right) \cup \left( \displaystyle { \bigcup_{\begin{array}{c}
 \mbox{\scriptsize $ (i,j)\in T(X),$ }  \\
\mbox{\scriptsize $ \theta \in G_{m^{(j)}}$ }   \end{array}}^{}} H^{(i,j)}_{\theta} \right) \cup \left( \displaystyle { \bigcup_{(i,j)\in NT(X)}^{}} H^{(i,j)} \right).$
\end{center}

\textit{Each term of the union above is a two dimensional plane in $\mathbb{C}^n$, repetitions may occur.}

\end{theorem}

  We note that the $C_5$-procedure is 
completely implementable on a computer program. In fact, another work on an implementation of this procedure in 
{\sc Singular} program \cite{singular} is being carried out by the second author and Aldicio Miranda (from Universidade Federal de Uberlândia - Brazil) and we hope it will be available soon.

\begin{example}\label{exe2}
Let $(X,0)=(X^{(1)}\cup X^{(2)} \cup X^{(3)} \cup X^{(4)},0)$ be the germ of curve in $(\mathbb{C}^3,0)$ where $(X^{(i)},0)$ is parametrized by the map $\varphi^i:(\mathbb{C},0)\rightarrow (X^{(i)},0)$, defined respectively by:

\begin{center}
$\varphi^{(1)}(u):= (u^6,u^{11}-u^9,u^{11}+u^9)$, $ \ \ \varphi^{(2)}(u):= (u^4,u^6,u^9) $, $ \ \ \varphi^{(3)}(u):= (u^4,u^3,u^5) \ \ $ and $ \ \ \varphi^{(4)}(u):= (u,2u,u)$.
\end{center}

Using the $C_5$-procedure we can see that $C_5(X,0)$ has seven different planes (see Figure \ref{figure1}) and one can easily find the reduced equations for each of them. Taking the product of the seven equations which define these planes, we find that

\begin{center}
$C_5(X,0)=V(xy^5z-y^5z^2-5xy^3z^3+5y^3z^4+4xyz^5-4yz^6)$.
\end{center}

To illustrate this, we present in Table \rm\ref{tabela1} \textit{all the auxiliary parametrizations of $(X,0)$ and the planes $H^{i}_{\theta}$, $H^{(i,j)}_{\theta}$ and $H^{(i,j)}$ used in the $C_5$-procedure to find $C_5(X,0)$. Also in Table} \rm\ref{tabela1} \textit{we present all auxiliary multiplicities of $(X,0)$ for completeness}.\\
 
 \begin{center}
 
 \end{center}

\begin{figure}[h]
\centering
\includegraphics[scale=0.22]{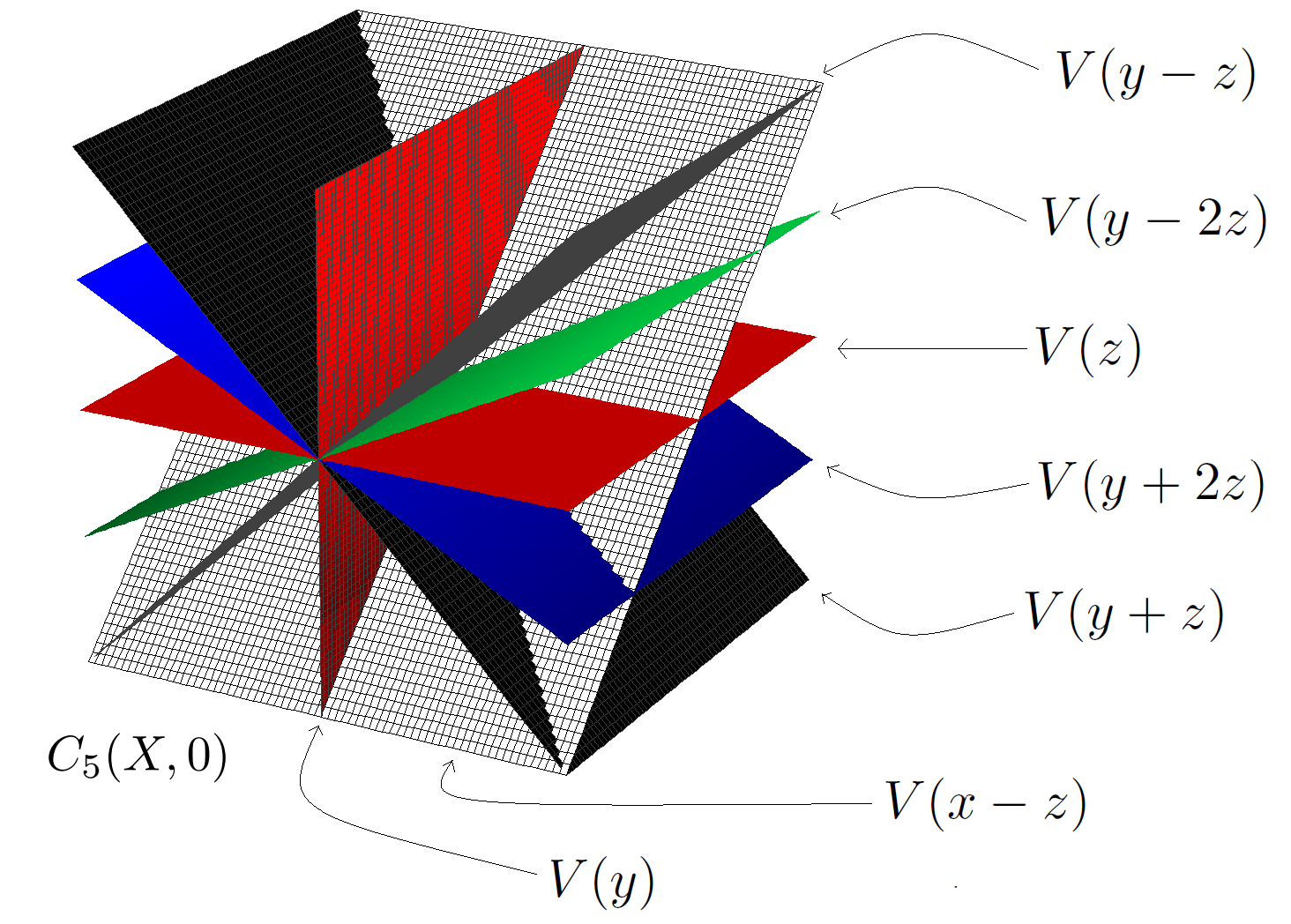}  
\caption{The cone $C_5(X,0)$ (real points).}\label{figure1}
\end{figure}

\newpage


\begin{table}[!h]
\caption{Auxiliary parametrizations and multiplicities of $(X,0)$ of Example \ref{exe2} }\label{tabela1}
\centering
{\def\arraystretch{2}\tabcolsep=12pt 
\begin{tabular}{c | c | l | c | c }

\hline
\multicolumn{4}{c}{\textbf{Step $1$}} \\
\hline			
$i \in S(X)$ &  $\theta$ & Characteristic auxiliary parametrization $\phi_{\theta}^{(i)}$  & $H_{\theta}^{(i)}$ & $m_{\theta}^{(i)}$ \\
  
\hline  
      &  $e^{\pi i / 3}$ & $(0,(\frac{1}{2}+\frac{i \sqrt{3} }{2})u^{11}-2u^9,(\frac{1}{2}+\frac{i \sqrt{3} }{2})u^{11}+2u^9)$ & $V(y+z)$ & $9$ \\
      
      &  $e^{2 \pi i / 3}$ & $(0,(\frac{3}{2}+\frac{i \sqrt{3} }{2})u^{11},(\frac{3}{2}+\frac{i \sqrt{3} }{2})u^{11})$ & $V(y-z)$ & $11$ \\
      
$i=1$ &  $-1$ & $(0,2u^{11}-2u^9,2u^{11}+2u^9)$ & $V(y+z)$ & $9$ \\
      
      &  $e^{4 \pi i / 3}$ & $(0,(\frac{3}{2}-\frac{i \sqrt{3} }{2})u^{11},(\frac{3}{2}-\frac{i \sqrt{3} }{2})u^{11})$ & $V(y-z)$ & $11$ \\
      
      &  $e^{5 \pi i / 3}$ & $(0,(\frac{1}{2}-\frac{i \sqrt{3} }{2})u^{11}-2u^9,(\frac{1}{2}-\frac{i \sqrt{3} }{2})u^{11}+2u^9)$ & $V(y+z)$ & $9$ \\
 
 \hline  
      &  $-1$ & $(0,0,2u^9)$ & $V(y)$ & $9$ \\

$i=2$ &  $i$ & $(0,2u^6,(1-i)u^9)$ & $V(z)$ & $6$ \\
   
      &  $-i$ & $(0,2u^6,(1+i)u^9)$ & $V(z)$ & $6$ \\     
  
  \hline 
  
 \multirow{2}{*}{$i=3$} &  $e^{\pi i / 3}$ & $(0,(\frac{3}{2}+\frac{i \sqrt{3} }{2})u^{4},(\frac{1}{2}+\frac{i \sqrt{3} }{2})u^{5})$ & $V(z)$ & $4$ \\
     
& $e^{2 \pi i / 3}$ & $(0,(\frac{3}{2}-\frac{i \sqrt{3} }{2})u^{4},(\frac{3}{2}+\frac{i \sqrt{3} }{2})u^{5})$ & $V(z)$ & $4$ \\
\hline
\multicolumn{4}{c}{\textbf{Step $2$}} \\
\hline			
$(i,j) \in T(X)$ &  $\theta$ & Contact auxiliary parametrization $\phi_{\theta}^{(i,j)}$  & $H_{\theta}^{(i,j)}$ & $m_{\theta}^{(i,j)}$ \\
  
\hline

\multirow{2}{*}{$(i,j)=(1,2)$} &  $e^{2k \pi i / 4}$, $k$ odd  & $(0, u^{44}, u^{36}+ u^{44}-\theta u^{54})$ & $V(y)$ & $36$ \\
     
& $e^{2k \pi i / 4}$, $k$ even & $(0,-2u^{36} + u^{44},u^{36}+ u^{44}-\theta u^{54})$ & $V(y+2z)$ & $36$ \\
\hline
\multicolumn{4}{c}{\textbf{Step $3$}} \\
\hline			
$(i,j) \in NT(X)$ &  $-$ & Contact auxiliary parametrization $\phi_{\theta}^{(i,j)}$  & $H^{(i,j)}$ & $m^{(i)}m^{(j)}$ \\
  
\hline

  $(i,j)=(1,3)$    &  $\theta \in G_3$ &  $(u^{18}-\theta u^{24},u^{33}-u^{27}-u^{18},u^{33}+u^{27}-\theta^2 u^{30})$ & $V(z)$ & $18$ \\
      
  $(i,j)=(1,4)$    &  $\theta \in G_1$ & $(0,u^{11}-u^9-2u^6,u^{11}+u^9-u^6)$ & $V(y-2z)$ & $6$ \\
      
  $(i,j)=(2,3)$    &  $\theta \in G_3$ & $(u^{12}-\theta u^{16},u^{18}-u^{12},u^{27} - \theta^2 u^{20})$ & $V(z)$ & $12$ \\
      
  $(i,j)=(2,4)$    &  $\theta \in G_1$ & $(0, u^6-2u^4,u^9-2u^4)$ & $V(y-2z)$ & $4$ \\
      
  $(i,j)=(3,4)$    &  $\theta \in G_1$ & $(u^4-u^3,-u^3,u^5-u^3)$ & $V(x-z)$ & $3$ \\
\hline
\end{tabular}
}
\end{table}

\textit{As we said in Remark} \rm\ref{coneremark}, \textit{we note that in this example:}

\begin{center}
$ \left( C_5(X^{(1)},0) \cup C_5(X^{(2)},0) \cup C_5(X^{(3)},0) \cup C_5(X^{(4)},0) \right) \subsetneqq C_5(X,0)$. 
\end{center} 

\end{example}

\begin{remark} Since in the non-tangent case $H^{(i,j)}_{\theta}$ does not depend on $\theta$, we do not need the contact auxiliary parametrizations in Step $3$ to find $H^{(i,j)}_{\theta}=H^{(i,j)}$. Eventhough the information of these auxiliary multiplicities is included in Table \rm\ref{tabela1}, \textit{they are not necessary for the $C_5$-procedure.  They are included for completeness, in order to illustrate other elements of Definition} \rm\ref{auxdef}.\\

\end{remark}


\begin{remark}\label{remark1}
Let $(X,0)$ be a germ of analytic reduced curve in $(\mathbb{C}^n,0)$ and let $(X^{(1)},0), \cdots, (X^{(r)},0)$ be its 
decomposition into irreducible components. Taking representatives $X^{(i)}$ of $(X^{(i)},0)$, consider sequences of points $(x_s), (y_s)$ in $X$ and a sequence of complex numbers $(\lambda_s)$ such 
that $(x_s),(y_s)\rightarrow 0$ 
and $\lambda_s\overline{(x_s-y_s)}\rightarrow v \neq 0$ as $s\rightarrow \infty$. After extracting sub-sequences if needed,  
we can assume there are only two possibilities: either both sequences are on the same branch, or they are in two different branches. 
Hence, we can reduce our study to the following three cases (each one of them corresponding to a step of Theorem \rm\ref{generalcase1}): 

\begin{flushleft}
\textit{Case (a): $(X,0)=(X^{(1)},0)\cup (X^{(2)},0)$ where $(X^{(1)},0)$ and $(X^{(2)},0)$ have different tangents.}\\
\textit{Case (b): $(X,0)$ is irreducible.}\\
\textit{Case (c): $(X,0)=(X^{(1)},0)\cup (X^{(2)},0)$ where $(X^{(1)},0)$ and $(X^{(2)},0)$ have the same tangent.}\\
\end{flushleft}

\end{remark}

\begin{proof}(of Theorem \ref{generalcase1}) The proof of this theorem is 
mainly due to Krasi\'nski in \cite{Krasinsky}, with ideas already present in 
\cite{briancon}. We describe below the link between the results of \cite{Krasinsky} and Theorem \ref{generalcase1}. The case where $(X,0)$ is a 
smooth branch has already been treated in Lemma \ref{smoothcase}.

Let $(X,0)=(X^{(1)}\cup \cdots \cup X^{(r)},0)$ be a germ of singular curve in $(\mathbb{C}^n,0)$.  As we pointed out in remark \ref{coneremark}
the $C_5$ cone of the curve is more than just the union of cones of its irreducible components. The other planes are the $H^{(i j)}_\theta$ which are the result
of taking sequences of points in two different irreducible components $(X^{(i)},0)$ and $(X^{(j)},0)$ of the curve. Krasi\'nsky refers to this as the relative tangent cone.

First of all, if the  irreducible components $(X^{(i)},0)$ and $(X^{(j)},0)$ have different tangents then \cite[Prop. 2.6]{Krasinsky} implies that the plane
$H^{(i,j)}$ generated by these two different tangent lines is the only extra component and so we get that
\[C_5(X^{(i)} \cup X^{(j)},0) = C_5(X^{(i)},0) \cup C_5(X^{(j)},0) \cup H^{(i,j)} \]

By theorem \ref{BGG} we know that each plane (irreducible component) of $C_5(X,0)$ contains at least one tangent line to $(X,0)$. For the case  
of a single branch, or two different branches with the same tangent, every such plane contains this tangent line. So to determine a plane of the $C_5$ cone
all we need is to compute a limit of secants not contained in the tangent line, and this is precisely the role of the auxiliary parametrizations.

If  $(X,0)$ is a singular branch then Theorem 3.4 in \cite{Krasinsky} states that:
\begin{center}
$C_5(X,0)=H_{\theta_1}\cup \cdots \cup H_{\theta_{m-1}}$, 
\end{center}
where $m$ is the multiplicity of $(X,0)$ and $1\neq \theta_i \in G_m$. 

On the other hand, if we have two irreducible components $(X^{(i)},0)$ and $(X^{(j)},0)$ with the same tangent, then again \cite[Th. 3.4]{Krasinsky} implies
that 
\[C_5(X^{(i)} \cup X^{(j)},0) = C_5(X^{(i)},0) \cup C_5(X^{(j)},0) \cup H^{(i,j)}_{\theta_1} \cup \cdots \cup H^{(i,j)}_{\theta_{m^{(j)}}}\]
with $m^{(j)} \leq m^{(i)}$ and $\theta_k \in G_{m^{(j)}}$.
\end{proof}

\section{Auxiliary multiplicities as a bi-Lipschtiz invariant}\label{sec4}

$ \ \ \ \ $ In this section we will show that the collection of all auxiliary multiplicities associated to a germ of curve $(X,0)$ in $(\mathbb{C}^n,0)$ is a complete invariant of the bi-Lipschitz type of the singularity (Theorem \ref{bilipequivalence}). In order to do that, we will see first in Section \ref{planecurves} that this is true for plane curves through some results of Pham and Teissier in \cite{pham}. After this, we will show in 
Section \ref{genericproj} that the auxiliary multiplicities of a curve characterize the generic plane projections. Finally, to conclude the proof, we will use in Section \ref{bilipschitzeq} a result by Teissier that says that the restriction to a curve $X \subset \mathbb{C}^n$ of a generic linear projection from $\mathbb{C}^n$ to $\mathbb{C}^2$ is bi-Lipschitz with respect to the outer metric. We note that the ideas and the spirit of this section are completely inspired by the paper \cite{briancon} and its appendix \cite{pham}.

\subsection{Plane curves}\label{planecurves}

$ \ \ \ \ $ In this section, unless otherwise stated, $(X,0)$ denotes a germ of plane curve. In \cite[Sec. 3]{pham}, Pham and Teissier proved a series of results on plane curves relating characteristic 
exponents and intersection multiplicities with auxiliary multiplicities (although the notion of ``auxiliary multiplicities'' does not appear in \cite{pham} with this terminology). For commodity of the reader we choose to rewrite some of these results in this work.  We will use these relations to state the following 
theorem, which is itself a reformulation of \cite[Prop. VI.3.2]{pham}.  

\begin{theorem}\label{theocurves} 
Let $(X,0)=(X^{(1)} \cup \cdots \cup X^{(r)},0)$ and $(Y,0)=(Y^{(1)} \cup \cdots \cup Y^{(r')},0)$ be the decomposition 
into branches of germs of plane curves. 
Then $(X,0)$ and $(Y,0)$ have the same topological type if and only if $r=r'$ and there exists a bijection 
$\rho: \lbrace 1,\cdots, r \rbrace \rightarrow \lbrace 1,\cdots, r \rbrace $ preserving all the auxiliary multiplicities, 
that is:
$$ChAM(X^{(i)},0)=ChAM(Y^{(\rho(i))},0) \ \ \ \, \, {\rm and}\, \, \ \ \  CoAM(X^{(i)},X^{(j)})=CoAM(Y^{(\rho(i))},Y^{(\rho(j))}),$$ 

\noindent for all $i\in S(X)$ and $(i,j)\in T(X) \cup NT(X)$.
 
\end{theorem}

The case of irreducible curves is taken care of by the following lemma:

\begin{lemma}\rm(\cite[\textit{Lem. VI}.3.3]{pham})\label{lemma1} \textit{Let $(X,0)$ be an irreducible plane curve. 
The set $ChAM(X,0)$ and the set of all characteristic exponents of $(X,0)$ are the same.}
\end{lemma}
   
      To give an idea of how this works, recall that if $m<\beta_1< \ldots < \beta_g$ is the set of characteristic exponents of the 
   germ $(X,0)$ with multiplicity $m$ then after a change of coordinates it has a parametrization $\varphi=(\varphi_1,\varphi_2)$ of the form:\\

\noindent $\hspace{1cm} \varphi_1(u)=u^m$ \\ 

\noindent $\hspace{1cm} \varphi_2(u) =u^{\beta_{1}}+\sum_{k=1}^{s_{1}} a_{\beta_{1}+k e_{1}} u^{\beta_{1}+k e_{1}}+a_{\beta_{2}} u^{\beta_{2}}+\sum_{k=1}^{s_{2}} a_{\beta_{2}+k e_{2}} u^{\beta_{2}+k e_{2}}+\cdots$ \\

$ \hspace{6cm} \cdots +a_{\beta_{j}} u^{\beta_{j}}+\sum_{k=1}^{s_{j}} a_{\beta_{j}+k e_{j}} u^{\beta_{j}+k e_{j}} +\cdots+a_{\beta_{g}} u^{\beta_{g}}+\sum_{k=1}^{\infty} a_{\beta_{g}+k} u^{\beta_{g}+k}$\\

\noindent where $e_1=gcd(m, \beta_1)$ and in general $e_j=gcd(e_{j-1},\beta_j)$. This gives rise to a filtration of the group of $m$-th roots of unity
   \[G_m \supset G_{e_1} \supset \cdots \supset G_{e_g}=\{1\} \]
   with the property that:
   \[ \theta \in  G_{e_{k-1}} \backslash G_{e_{k}} \Longleftrightarrow m_\theta:= ord_0(\varphi_2(u)-\varphi_2(\theta u))=\beta_{k}\]
   (for a more detailed explanation see \cite[Sec. 8.1.2]{GiSiTe}) which gives the desired result.\\

 For the general case it is enough to consider a curve with two branches $(X,0)=(X^{(1)} \cup X^{(2)},0)$. In this setting the topological type of the curve
 determines and is determined by the set of characteristic exponents for each branch, plus the intersection multiplicity between them. The following
 lemma tells us that this intersection multiplicity can be recovered from $CoAM(X^{(1)}, X^{(2)},0)$, the ordered sequence of contact auxiliary multiplicities of $(X^{(1)}\cup X^{(2)},0)$.
 
\begin{lemma}\rm(\cite[\textit{Lem. VI}.3.4]{pham}) \label{lemma2} \textit{The intersection multiplicity $i(X^{(1)},X^{(2)},0)$ of  two plane branches $(X^{(1)},0)$ and 
$(X^{(2)},0)$ at $0$ satisfies:
$$i(X^{(1)},X^{(2)})=\dfrac{(m_{\theta_1}^{(1,2)}+ \cdots + m_{\theta_{m^{(2)}}}^{(1,2)})}{m^{(2)}},$$
\noindent where $m^{(1)} \geq m^{(2)}$ and $G_{m^{(2)}}= \lbrace {\theta_1}^{}, \cdots , {\theta_{m^{(2)}}} \rbrace $, (see Definition} \ref{auxdef} b)).

\end{lemma}

\begin{remark}\label{remarklemma2} Note that:
\begin{enumerate}
  \item  If the branches are transversal then  $m_{\theta}^{(1,2)}=m^{(1)}m^{(2)}$ for all $\theta \in G_{m^{(2)}}$ and so 
    $$i(X^{(1)},X^{(2)})= \dfrac{m^{(1)}(m^{(2)})^2}{m^{(2)}}=m^{(1)}m^{(2)}.$$
    \item This lemma is practically a direct  consequence of Halphen's formula \rm\cite[\textit{Prop}. 3.10]{greuel6} \textit{as seen in the proof by Pham and Teissier}
    in \cite[Lem. VI.3.4]{briancon}.
   \end{enumerate} 
\end{remark}

The last ingredient for the proof of Theorem \ref{theocurves} is to understand how the topological type of the plane curve  $(X,0)=(X^{(1)}\cup X^{(2)},0)$ 
determines all of its auxiliary multiplicities. We already know (Lemma \ref{lemma2})  that $ChAM(X^{i},0))$ corresponds to the set of characteristic exponents
of the branch. Also, if the branches are transversal, the previous remark tells us that the intersection multiplicity determines the contact auxiliary multiplicities $CoAM(X^{(1)}, X^{(2)},0)$, so all we are left with is the tangent case.  \\

Lemma \ref{lemma3} below gives us a precise description of  $CoAM(X^{(1)}, X^{(2)},0)$. But most importantly for us, it establishes that this ordered 
sequence of integers is an invariant of the topological type of the curve. Let

\begin{center}
 $\varphi^{(1)}(u)=\left(\displaystyle u^{m^{(1)}} \ , \displaystyle { \sum_{i\geq m^{(1)}}^{}}a_{i}u^i \right)$ $ \ \ \ $ and $ \ \ \ $ $\varphi^{(2)}(u)=\left(\displaystyle u^{m^{(2)}} \ , \displaystyle { \sum_{i\geq m^{(2)}}^{}}b_{i}u^i \right)$
 \end{center} 

\noindent be a compatible system of parametrizations for $(X,0)$. Recall that  $ m_{\theta}^{(1,2)}$  is defined as the order of the series
 \[ \varphi_2^{(1)}( u^{ m^{(2)}})- \varphi_2^{(2)}(\theta u^{m^{(1)}}) \]
 Note that under this reparametrization of the branches, the sequence of  characteristic exponents of the branch $(X^{(1)},0)$ is multiplied by $m^{(2)}$
 and the sequence of  characteristic exponents of the branch $(X^{(2)},0)$ is multiplied by $m^{(1)}$.
 We will assume that these new sequences coincide up to order $\tau \geq 0$. Meaning that if  $m^{(1)}<\alpha_1< \cdots < \alpha_s$ and
  $m^{(2)}<\gamma_1 < \cdots < \gamma_r$ are the characteristic exponents of $(X^{(1)},0)$ and $(X^{(2)},0)$, respectively, then 
  \begin{center}
$\alpha_1 m^{(2)}= \gamma_1 m^{(1)} =:\beta_1$, $ \ \ $ $\alpha_2 m^{(2)}= \gamma_2 m^{(1)}
=:\beta_2$, $ \ \ $  $ \ \cdots \ $ $ \ \ $, $\alpha_\tau m^{(2)} = \gamma_\tau m^{(1)} =: \beta_\tau$. 
\end{center}
Finally, let $\delta:= max \lbrace m_{\theta}^{(1,2)} \ | \ \theta \in G_{m^{(2)}} \rbrace$. Now we are able to state: \\

\begin{lemma}(\rm\cite{pham}, \textit{Lem. VI.3.5) \label{lemma3}} 
\textit{There exists an integer $q$, determined by the intersection multiplicity of the branches and by the characteristic exponents of $(X^{(1)},0)$ and $(X^{(2)},0)$, such that
 $0 \leq q \leq \tau$  and:}\\

\noindent \textit{(a) The set of contact auxiliary multiplicities is equal to          
         \[ \lbrace \beta_1,\cdots, \beta_q, \delta \rbrace \]
         (In case $q=0$ it is only $\delta$.)}\\

\noindent \textit{(b) Set $e_0=m^{(2)}$.  For $j \in \{1,\ldots \tau\}$ define $e_j=gcd\left (m^{(2)},\gamma_1,\ldots \gamma_j\right)$ , then the number of $\theta$'s in $G_{m^{(2)}}$ such that $m^{(1,2)}_\theta=\beta_k$ (respect. $m^{(1,2)}_\theta= \delta$) is equal to $e_{k-1}-e_{k}$ (respect. $e_q$).}

\end{lemma}

 These two data combined give us $CoAM(X^{(1)}, X^{(2)},0)$. More precisely, the ordered sequence of contact auxiliary multiplicities  is:
   
\[  \begin{array}{ccccccc}
\underbrace{\beta_1=\cdots=\beta_1} & < & \underbrace{\beta_2=\cdots=\beta_2} & < \ \cdots \ < & \underbrace{\beta_q=\cdots=\beta_q} & < & \underbrace{\delta=\cdots=\delta} \\

e_0-e_1 \ times &  & e_1-e_2 \ times  &  & e_{q-1}-e_q \ times &  & e_q \ times 
 \end{array} \]

\begin{example}
 Let $(X,0) \subset (\C^2,0)$  be a reduced curve with two tangent branches parametrized by
\begin{align*}
  \varphi^{(1)}(u)&= \left( u^{12}, u^{18}+u^{33}+u^{34}\right) \\
  \varphi^{(2)}(u)&=\left( u^8, u^{12}+u^{20}+2u^{22} + u^{23} \right)
\end{align*} 
  Note that $m^{(2)}=8$, in order to compute the contact auxiliary multiplicities we need to compute the order of the series
  \[ \varphi^{(1)}(u^8)-\varphi^{(2)}( \theta u^{12}) 
  =\left(0, (1-\theta^4)u^{144} -\theta^4 u^{240} + (1-\theta^{6})u^{264} + u^{272} - \theta^7 u^{276}\right)\]
  with $\theta \in G_8$. Following the notation of Lemma \rm\ref{lemma3} \textit{we have $\tau=2$ with}
  \[ \beta_1=144 \hspace{0.5in} \textrm{and} \hspace{0.5in} \beta_2=264\]
  \[e_0=8, \ \  \;   e_1=gcd(8,12)=4 \ \ \; and \ \ \ e_2=gcd(8,12,22)=2.\]
  \textit{Now, for $\theta \in G_8 \setminus G_4$ we have that $m_{\theta}^{(1,2)}=\beta_1=144$, and for $\theta \in G_4$ 
  we have that $m_{\theta}^{(1,2)}=240$. This implies that $q=1$, $\delta=240$ and we have $e_0-e_1=4$ $\theta$'s that give us $\beta_1$ and
   $e_1=4$ $\theta$'s that give us $\delta$. Finally, we can easily calculate the intersection multiplicity between $(X^{(1)},0)$ and $(X^{(2)},0)$ as:}
   \[i(X^{(1)},X^{(2)})=\frac{144\cdot 4 + 240\cdot 4}{8} =192\]
   
\end{example}

Combining the three previous lemmas, we can proceed to prove Theorem \ref{theocurves}.\\

\begin{proof}(of Theorem \ref{theocurves}) Consider the plane curves $(X,0)=(X^{(1)} \cup \cdots \cup X^{(r)},0)$ and 
$(Y,0)=(Y^{(1)} \cup \cdots \cup Y^{(r')},0)$. Suppose that $r=r'$ and also that there is a bijection 
$\rho: \lbrace 1,\cdots, r \rbrace \rightarrow \lbrace 1,\cdots, r \rbrace $ preserving all the auxiliary multiplicities. 
By Lemmas \ref{lemma1} and \ref{lemma2}, the bijection $\rho$ also preserves the characteristic exponents and the 
intersection multiplicities of $(X,0)$. Hence, $(X,0)$ and $(Y,0)$ have the same topological type.

Conversely, if $(X,0)$ and $(Y,0)$ have the same topological type, then $r=r'$ and there exists a bijection $\rho$ on the set $\lbrace 1,\cdots, r \rbrace$ preserving characteristic exponents 
and intersection multiplicities. It is enough to consider the case where $r=2$, i.e, $(X,0)=(X^{(1)}\cup X^{(2)},0)$ and $(Y,0)=(Y^{(1)}\cup Y^{(2)},0)$.

Let $\varphi^{(1)}$ and $\varphi^{(2)}$ (respectively, $\psi^{(1)}$ and $\psi^{(2)}$) be a system of compatible Puiseux parametrizations for $(X,0)$ (respectively, $(Y,0)$). Let $m_{\theta}^{(1,2)}$ (respectively, $\textbf{m}_{\theta}^{(1,2)}$) be a contact auxiliary multiplicity of the pair $(X^{(1)}\cup X^{(2)},0)$ (respectively, $(Y^{(1)}\cup Y^{(2)},0)$). We have that $ m_{\theta}^{(1,2)}$ (respectively, $\textbf{m}_{\theta}^{(1,2)}$)  is defined as the order of the series

\begin{center}
$\varphi_2^{(1)}( u^{m^{(2)}})- \varphi_2^{(2)}\left(\theta u^{m^{(1)}}\right)$, $ \ \ \ $ (respectively, $ \ \psi_2^{(1)}( u^{ \textbf{m}^{(2)}})- \psi_2^{(2)}\left(\theta u^{\textbf{m}^{(1)}}\right)$.
\end{center}

 Note that under this reparametrization of the branches, the sequence of  characteristic exponents of the branch $(X^{(1)},0)$ (respectively, $(Y^{(1)},0)$) is multiplied by $m^{(2)}$ (respectively, $\textbf{m}^{(2)}$), and the sequence of  characteristic exponents of the branch $(X^{(2)},0)$ (respectively, $(Y^{(2)},0)$) is multiplied by $m^{(1)}$ (respectively, $\textbf{m}^{(1)}$).
 We will assume that these new sequences coincide up to order $\tau \geq 0$ (respectively, $\tau'\geq 0$). Meaning that if  $m^{(1)}<\alpha_1< \cdots < \alpha_s$ and $ m^{(2)} < \gamma_1 < \cdots < \gamma_l $ (respectively, $ \textbf{m}^{(1)} < \alpha_1' < \cdots < \alpha_{s'}'$ and
  $\textbf{m}^{(2)}<\gamma_1' < \cdots < \gamma_{l'}'$) are the characteristic exponents of $(X^{(1)},0)$ and $(X^{(2)},0)$ (respectively, $(Y^{(1)},0)$ and $(Y^{2},0)$), then:
   
  \begin{center}
$\alpha_1 m^{(2)}= \gamma_1 m^{(1)} =:\beta_1$, $ \ \ $ $\alpha_2 m^{(2)}= \gamma_2 m^{(1)}
=:\beta_2$, $ \ \ $  $ \ \cdots \ $ $ \ \ $, $\alpha_\tau m^{(2)} = \gamma_\tau m^{(1)} =: \beta_\tau$.\\

$ \ \ $\\

(respectively, $ \ \ $ $\alpha_1' \textbf{m}^{(2)}= \gamma_1' \textbf{m}^{(1)} =:\beta_1'$, $ \ \ $ $\alpha_2' \textbf{m}^{(2)}= \gamma_2' \textbf{m}^{(1)}
=:\beta_2'$, $ \ \ $  $ \ \cdots \ $ $ \ \ $, $\alpha_{\tau'}' \textbf{m}^{(2)} = \gamma_{\tau'}' \textbf{m}^{(1)} =: \beta_{\tau'}'$).
\end{center}

Lemma \ref{lemma1} implies that the set of characteristic auxiliary multiplicities of $(X^{(i)},0)$ and $(Y^{(i)},0)$ are the same for all $i$. In particular, by Lemma \ref{lemma1} we have that $\tau = \tau'$ and $\beta'_i=\beta_i$ for all $i\leq \tau$.

 Using the notation of Lemma \ref{lemma3}, we have that there exists an integer $q\geq 0$ (respectively, $q'\geq0$) such that the set of contact auxiliary multiplicities of the pair $(X^{(i)}\cup X^{(j)},0)$  (respectively, $(Y^{(i)}\cup Y^{(j)},0)$) is 
either $ \lbrace \delta \rbrace $ if $q=0$ (respectively, $\delta'$ if $q'=0$), or $\lbrace \beta_1,\cdots, \beta_q, \delta \rbrace$ if $q\geq 1$ (respectively, $\lbrace \beta_1,\cdots, \beta_{q'}, \delta' \rbrace$ if $q'\geq 1$).
 
Again by Lemma \ref{lemma3}, $q$ (respectively, $q'$) is determined by the intersection multiplicity and by the characteristic exponents of $(X^{(1)},0)$ and $(X^{(2)},0)$ (respectively, $(Y^{(1)},0)$ and $(Y^{(2)},0)$). Therefore, $q'=q$ and using Lemma \ref{lemma2} and Lemma \ref{lemma3}(b) and (c) one can conclude that $\delta'=\delta$. Therefore the ordered sequence of contact auxiliary multiplicities of $(X^{(1)},0)$ and  $(X^{(2)},0)$ and 
the one of $(Y^{(1)},0)$ and  $(Y^{(2)},0)$ are the same.\end{proof}

\begin{remark}\label{gooddef1}
\textit{Note that Theorem} \rm\ref{theocurves} \textit{and Lemmas} \rm\ref{lemma1}, \ref{lemma2} \textit{and} \rm\ref{lemma3} \textit{show in particular that for a plane curve $(X,0)$ the auxiliary multiplicities do not depend on the choice of the parametrization of the branches of $(X,0)$.}
\end{remark}

\subsection{Generic projections of space curves in $\mathbb{C}^2$}\label{genericproj}

$ \ \ \ \ $ Now we will prove that the auxiliary multiplicities of a space curve singularity are the same as those of its generic 
projection to $\mathbb{C}^2$. In this context the genericity we consider is characterized by transversality of the direction 
of the projection with the $C_5$-cone of the space curve singularity. More precisely:

\begin{definition}\label{defgenericproj}
Let $(X,0)$ be a germ of curve in $(\mathbb{C}^n,0)$. Let $\pi:(\mathbb{C}^n,0)\rightarrow (\mathbb{C}^2,0)$ be a linear projection. 
We say that the projection $\pi$ is $C_5$-generic for the germ of curve $(X,0)$ if the kernel of $\pi$ intersects transversally 
the cone $C_5(X,0)$; that is, $ker(\pi)\cap C_5(X,0)=\{0\}$. 

When $\pi$ is $C_5$-generic for $(X,0)$, we will denote by $(\tilde{X},0)$ the image $\pi(X,0)$, and we will call it the 
$C_5$-generic projection of $(X,0)$ \rm(\textit{see} \rm\cite[\textit{Chap.} IV]{briancon}).
\end{definition}

\begin{remark}\label{preserves tangency}
We will see in the proof of Proposition \rm\ref{propgeneric} \textit{that a $C_5$-generic projection in the sense of Definition} \rm\ref{defgenericproj}, \textit{preserves multiplicities of the branches, that is: a curve and its $C_5$-generic projection have the same multiplicity.}

\textit{We can also notice that if two branches of a curve are non-tangent then neither are the branches of its $C_5$-generic projection. In fact, if the images by a linear projection of two non-tangent branches are tangent plane curves, 
then the two dimensional plane, generated by the tangent lines of both branches, intersects the kernel of the projection. 
We have seen in Theorem} \rm\ref{generalcase1}, \textit{that 
the plane generated by both tangents is a plane of the $C_5$-cone of the space curve. So the projection is not $C_5$-generic.}   
\end{remark}

\begin{proposition}\label{propgeneric} Let $(X,0)$ be a germ of curve in $\mathbb{C}^n$, $\pi:(\mathbb{C}^n,0)\rightarrow (\mathbb{C}^2,0)$ be a linear projection. Then, $\pi$ is $C_5$-generic if and only if

\begin{center}
 $ChAM(X,0)=ChAM(\pi(X),0) $ $ \ \ \ $ and $ \ \ \ $ $CoAM(X,0)=CoAM(\pi(X),0)$.
\end{center}

\end{proposition}

\begin{proof} ($\Rightarrow$) We will suppose first that $(X,0)$ is irreducible. Let $\varphi$ be a Puiseux parametrization of $(X,0)$:

\begin{center}
$\varphi(u):= \left( u^m, \displaystyle { \sum_{i> m}^{}}a_{2,i}u^i \ , \  \displaystyle { \sum_{i> m}^{}}a_{3,i}u^i, \   
\cdots \ , \ \displaystyle { \sum_{i>m}^{}}a_{n,i}u^i \right)$. 
\end{center}

\noindent Then, for $\theta \neq 1 \in G_m$, the auxiliary parametrization $\phi_{\theta}$ has the form:

\begin{center}
$\phi_{\theta}(u)= \left(\displaystyle {0, a_{2,k_{\theta}}(1-\theta^{k_{\theta}})u^{k_{\theta}}+ 
\sum_{i> k_{\theta}}^{}}a_{2,i}(1-\theta^i)u^i \ , \   \cdots \ , \ a_{n,k_{\theta}}(1-\theta^{k_{\theta}})u^{k_{\theta}} 
+ \displaystyle { \sum_{i>k_{\theta}}^{}}a_{n,i}(1-\theta^i)u^i \right)$, 
\end{center}

\noindent where $k_{\theta}= \ min \lbrace \ i \ | \ $  there exists $j$ with $ a_{j,i}(1-\theta^i)\neq 0 \rbrace$. 
In particular, $a_{j,k_{\theta}} \neq 0$ for some $j$, and therefore the characteristic auxiliary multiplicity  
$m_{\theta}=k_{\theta}$. 

Now consider a projection $\pi: (\mathbb{C}^n,0)\rightarrow (\mathbb{C}^2,0)$ defined by 
$\pi(x_1,x_2,\cdots,x_n):=(x_1,\lambda_2 x_2+\cdots \lambda_n x_n)$. Since $C_3(X,0)$ is generated by $(1,0, \ldots, 0)$ and 
$C_3(X,0) \subset C_5(X,0)$, for generic values of $\lambda_2,\cdots,\lambda_n$ the projection $\pi$ is $C_5$-generic for $(X,0)$. 

Suppose $\pi$ is $C_5$-generic for $(X,0)$, then $\pi$ induces a Puiseux parametrization for the $C_5$-generic projection $(\tilde{X},0)$ 
given by $\tilde{\varphi}= \pi \circ \varphi$. Furthermore, the auxiliary parametrization $\tilde{\phi}_{\theta}$ associated 
to $(\tilde{X},0)$ has the form:

\begin{center}
$\tilde{\phi}_{\theta}(u)= \left(\displaystyle {0, (\lambda_2 a_{2,k_{\theta}}+ \cdots + \lambda_n a_{n,k_{\theta}})(1-\theta^{k_{\theta}})u^{k_{\theta}}+ \sum_{i> k_{\theta}}^{}} (\lambda_2 a_{2,i}+ \cdots + \lambda_n a_{n,i})(1-\theta^i)u^i \right).$ 
\end{center}

\noindent Note that if $(\lambda_2 a_{2,k_{\theta}}+ \cdots + \lambda_n a_{n,k_{\theta}})\neq 0$, then the order of the second 
coordinate of $\tilde{\phi}_{\theta}$ is $k_{\theta}$ which is then the characteristic auxiliary multiplicity associated to $(\tilde{X},0)$ 
and $\theta$, and the result follows. Notice that in particular, the multiplicity of a curve and its $C_5$-generic projection are 
the same, as claimed in Remark \ref{preserves tangency}.

Let us now prove that $(\lambda_2 a_{2,k_{\theta}}+ \cdots + \lambda_n a_{n,k_{\theta}})$ is in fact non-zero. Suppose 
$(\lambda_2 a_{2,k_{\theta}}+ \cdots + \lambda_n a_{n,k_{\theta}})=0$, this means that the line generated by the vector 
$ v_{\theta}=(0, a_{2,k_{\theta}}, \cdots , a_{n,k_{\theta}})$ is in $ker(\pi)$. 
Note that the vector $v_{\theta}$ generates the cone $C_3(A_{\theta}(X,0))$ of the auxiliary characteristic curve $A_{\theta}(X,0)$.  By Theorem \ref{generalcase1}, $v_{\theta} \in C_5(X,0)$, which contradicts that $\pi$ is $C_5$-generic for $(X,0)$.

For the case where $(X^{(1)},0)$ and $(X^{(2)},0)$ are not tangent, we saw in Remark \ref{preserves tangency} that 
their images by a $C_5$-generic projection are not tangent and the multiplicities are the same. Thus the proof follows by Lemma \ref{lemma2}. The proof of the case of two tangent branches is analogous to the proof of the irreducible case.\\

$(\Leftarrow)$ Suppose that $ChAM(X,0)=ChAM(\pi(X),0)$ and  $CoAM(X,0)=CoAM(\pi(X),0)$ holds for a non-$C_5$-generic linear projection $\pi$. Then by transitivity, we have that $ChAM(\pi(X),0)=ChAM(\tilde{X},0)$ and  $CoAM(\pi(X),0)=CoAM(\tilde{X},0)$. In particular, we have the equality between Milnor numbers,

\begin{center}
$\mu(\pi(X),0)=\mu(\tilde{X},0)$,
\end{center}

\noindent which is a contradiction since $\mu(\tilde{X},0)<\mu(\pi(X),0)$ by \cite[Prop. IV.2]{briancon} or \cite[Prop. 8.4.6]{GiSiTe}.\end{proof}

\begin{remark}\label{gooddef2} Since all $C_5$-generic projections of a germ of curve have the same topological type 
(see \rm\cite[\textit{Prop}. 8.4.6]{GiSiTe} \textit{or} \rm\cite[\textit{Prop}. IV.2]{briancon}) \textit{then by Proposition} \rm\ref{propgeneric} \textit{and Remark} 
\rm\ref{gooddef1} \textit{we have that the auxiliary multiplicities do not depend on the choice of the 
compatible system of Puiseux parametrizations of $(X,0)$.}
\end{remark}

\subsection{Bi-Lipschitz equivalence between curves}\label{bilipschitzeq}

$ \ \ \ \ $ In this section we study the bi-Lipschitz equivalence between germs of curves. We will consider only the outer metric, i.e., the metric induced by the one on the ambient space.
We say that two germs of curves $(X,0)$ and $(Y,0)$ in $(\mathbb{C}^n,0)$ are bi-Lipschitz equivalent if there exists a 
germ of bi-Lipschitz homeomorphism $h:(X,0)\rightarrow (Y,0)$ (that is, $h$ and $h^{-1}$ are Lipschitz maps).

As a consequence of the previous results, we will prove in this section that the auxiliary multiplicities can be used to 
determine when two curves in $(\mathbb{C}^n,0)$ are bi-Lipschitz equivalent.

\begin{theorem}\label{bilipequivalence} Let $(X,0)=(X^{(1)} \cup \cdots \cup X^{(r)},0) \subset (\mathbb{C}^n,0)$ and $(Y,0)=(Y^{(1)} \cup \cdots \cup Y^{(r')},0) \subset (\mathbb{C}^{n'},0)$ be germs of curves. Then $(X,0)$ and $(Y,0)$ are bi-Lipschitz equivalent if and only if $r=r'$ and there exists a bijection 
$\rho: \lbrace 1,\cdots, r \rbrace \rightarrow \lbrace 1,\cdots, r' \rbrace $ preserving all the auxiliary multiplicities, that is:

\begin{center}
$\text{ChAM}(X^{(i)},0))=\text{ChAM}(Y^{(\rho(i))},0)$,  $  \ \ \ \ $   and 
$ \ \ \ \ $  $\text{CoAM}(X^{(i)},X^{(j)})=\text{CoAM}(Y^{(\rho(i))},Y^{(\rho(j))})$,
\end{center} 

\noindent for all $i\in S(X)$ and $(i,j)\in T(X)$ or $(i,j) \in NT(X)$. 
\end{theorem}

\begin{proof} Let $\pi_1:(X,0)\rightarrow (\tilde{X},0)$ and $\pi_2:(Y,0)\rightarrow (\tilde{Y},0)$ be $C_5$-generic projections of $(X,0)$ and $(Y,0)$, respectively. Note that $\pi_1$ and $\pi_2$ are bi-Lipschitz homeomorphisms (\cite{Teissier2}, see also \cite{pichon}, Th. $5.1$).

Suppose that $(X,0)$ and $(Y,0)$ are bi-Lipschitz equivalent and let $h:(X,0)\rightarrow (Y,0)$ be a bi-Lipschitz homeomorphism, so $r=r'$. Note that the map $\pi_2 \circ h \circ \pi_1^{-1}:(\tilde{X},0)\rightarrow (\tilde{Y},0)$ is a bi-Lipschitz homeomorphism. By Theorem \ref{theocurves} and Proposition \ref{propgeneric} we can construct a bijection $\rho: \lbrace 1,\cdots, r \rbrace \rightarrow \lbrace 1,\cdots, r \rbrace $ preserving all the auxiliary multiplicities.

Suppose now that $r=r'$ and there exists a bijection $\rho: \lbrace 1,\cdots, r \rbrace \rightarrow \lbrace 1,\cdots, r \rbrace $
preserving all the auxiliary multiplicities. By Proposition \ref{propgeneric} and Theorem \ref{theocurves}
we have that $(\tilde{X},0)$ and $(\tilde{Y},0)$ have the same topological type. 
Since they are plane curves there exists a bi-Lipschitz homeomorphism $\tilde{h}:(\tilde{X},0)\rightarrow (\tilde{Y},0)$. 
Note that the map $\pi_{2}^{-1} \circ \tilde{h} \circ \pi_{1}:(X,0)\rightarrow (Y,0)$ is a bi-Lipschitz homeomorphism, 
therefore $(X,0)$ and $(Y,0)$ are bi-Lipschitz equivalent, as desired.\end{proof}

\section{Upper bounds}\label{bounds}

$ \ \ \ \ $ In this section we present some upper bounds for the number of irreducible components of $C_5(X,0)$. \\

For a finite subset $A$ of $\mathbb{N}$ or $ \mathbb{N}^2$ the notation $\sharp \lbrace A \rbrace$ means the number of elements of $A$. For an analytic set (or germ of analytic set) $W \subset \mathbb{C}^n$ the notation $\sharp \lbrace Irred(W) \rbrace$ means the number of irreducible components of $W$ (or the germ $(W,0)$). As a direct consequence of Theorem \ref{generalcase1} we have the following bound: 

\begin{center}
 $\sharp \lbrace Irred(C_5(X,0)) \rbrace \leq  \left(  \displaystyle { \sum_{i\in S(X)}^{}} (m^{(i)}-1) \right) +  \left( \displaystyle { \sum_
 {\begin{array}{c} 
 {\footnotesize (i,j)\in T(X)} \\
  {\footnotesize m^{(j)} \leq m^{(i)}}
  \end{array}}} m^{(j)} \right) + \sharp \lbrace NT(X) \rbrace$.
\end{center}

  However, Krasi\'nsky proved in \cite[Cor. 3.6]{Krasinsky} that for a pair of tangent branches $(X,0)$ and $(Y,0)$ the amount of planes of the $C_5$-cone arising
  from taking sequences in different branches is bounded above by the greatest common divisor of the multiplicities of the branches. This gives us the following
  bound. 

\begin{corollary}\label{bound1} \textbf{(Upper bound $1$)} Let $(X,0)=(X^1\cup \cdots \cup X^r,0)$ be a germ of singular curve in $(\mathbb{C}^n,0)$. Then:

\begin{center}
 $\sharp \lbrace Irred(C_5(X,0)) \rbrace \leq  \left(  \displaystyle { \sum_{i\in S(X)}^{}} (m^{(i)}-1) \right) +  \left( \displaystyle { \sum_{(i,j)\in T(X)}^{}} 
 \textrm{gcd}\left(m^{(i)}, m^{(j)} \right) \right) + \sharp \lbrace NT(X) \rbrace$.
\end{center}

\noindent In particular, if $(X,0)$ is irreducible, then $\sharp \lbrace Irred(C_5(X,0)) \rbrace \leq m(X,0)-1$.
\end{corollary}


In the irreducible case, the following lemma shows both, this upper bound and the $C_5$-procedure can be improved.

\begin{lemma}\label{ordroots} Let $(X,0)$ be an irreducible germ of singular curve with multiplicity $m$ and take $\theta_{p},\theta_q \in G_m \setminus \lbrace 1 \rbrace$. If $ord(\theta_p)=ord(\theta_q)$, then the planes $H_{\theta_p}$ and $H_{\theta_q}$ are the same (as complex vector spaces).
\end{lemma}

\begin{proof} Let ${\varphi}$ be a Puiseux parametrization of $({X},0)$ defined by

\begin{center}
${\varphi}(u):= \left( u^m, \displaystyle { \sum_{i> m}^{}}a_{1,i}u^i \ , \  \displaystyle { \sum_{i> m}^{}}a_{2,i}u^i, \   \cdots \ , \ \displaystyle { \sum_{i>m}^{}}a_{n,i}u^i \right)$. 
\end{center}

\noindent For $\theta_p \in G_m \setminus \{1\} $, the auxiliary parametrization $\phi_{\theta_p}$ has the following form:

\begin{center}
$\phi_{\theta_p}(u)= \left(\displaystyle {0, a_{2,k_{p}}(1-\theta_p^{k_{p}})u^{k_{p}}+ \sum_{i> k_{p}}^{}}a_{2,i}(1-\theta_p^i)u^i \ , \   \cdots \ , \ a_{n,k_{p}}(1-\theta_p^{k_{p}})u^{k_{p}} + \displaystyle { \sum_{i>k_{p}}^{}}a_{n,i}(1-\theta_p^i)u^i \right)$, 
\end{center}

\noindent where $k_{p}= \ min \lbrace \ i \ | \ $  there exists $j$ with $ a_{j,i}(1-\theta_p^i)\neq 0 \rbrace$. In particular, $a_{j,k_{p}} \neq 0$ for some $j$. Hence, the $C_3(A_{\theta_p}(X),0)$ is generated by the vector $v_{\theta_p}:=(0,a_{2,k_{p}}(1-\theta_p^{k_{p}}), \cdots, a_{n,k_{p}}(1-\theta_p^{k_{p}}))$. Since $ord(\theta_p)=ord(\theta_q)$, $\theta_{p}^s=1$ if and only if $\theta_{q}^s=1$. Therefore, we have that

\begin{center}
$k_p=min \lbrace \ i \ | \ $ $\exists j$ with $ a_{j,i}\neq 0$ and $\theta_p^i\neq 1 \rbrace$ $ = $ $min \lbrace \ i \ | \ $ $\exists j$ with $ a_{j,i}\neq 0$ and $\theta_q^i\neq 1 \rbrace$. 
\end{center}

\noindent Hence, the vectors $v_{\theta_p}$ and $v_{\theta_q}$ have the same direction. It follows that $H_{\theta_p}$ and $H_{\theta_q}$ are the same plane.
\end{proof} \\

 For our next bound we will use the following notation. We say that a sequence of positive integers $d_1=n,d_2,d_3,\cdots, d_s= 1$, such that $d_{i+1}$ divides $d_i$ is a sequence of nested divisors of $n$. 

Consider the function $ \sigma :\mathbb{N} \setminus \lbrace 0 \rbrace \rightarrow \mathbb{N} \setminus \lbrace 0 \rbrace$ defined as 

\begin{center}
$\sigma(n)=$ the maximum length of sequences of nested divisors of $n$. 
\end{center}

\begin{remark}\label{nested}
\noindent If $(X,0)$ is a germ of irreducible plane curve of multiplicity $m$,
then $\sigma(m)$ is the maximum number of characteristic exponents that $(X,0)$ can have. Recall that if $\theta \in G_m$ and $ord(\theta)=d$, then $d$ is a (positive) divisor of $m$. Since $G_m$ is a cyclic group, the converse is also true, that is, if $d$ is a positive divisor of $m$, then there exists $\theta \in G_m$ such that $ord(\theta)=d$. Hence, we have another upper bound for $\sharp \lbrace Irred(C_5(X,0)) \rbrace$ which improves Upper bound $1$.

\end{remark}

\begin{proposition}\label{bound2} \textbf{(Upper bound $2$)} Let $(X,0)=(X^1 \cup \cdots \cup X^r,0)$ be a germ of singular curve in $(\mathbb{C}^n,0)$. Then

\begin{center}
 $\sharp \lbrace Irred(C_5(X,0)) \rbrace \leq  \left(  \displaystyle { \sum_{i\in S(X)}^{}} \sigma(m^{(i)})-1 \right) +  \left( \displaystyle { \sum_{(i,j)\in T(X)}^{}} 
  \textrm{gcd}\left(m^{(i)}, m^{(j)} \right) \right) + \sharp \lbrace NT(X) \rbrace$.
\end{center}

\noindent In particular, if $(X,0)$ is irreducible, then $\sharp \lbrace Irred(C_5(X,0)) \rbrace \leq \sigma(m(X,0))-1$.
\end{proposition}

\begin{proof} The proof follows by Corollary \ref{bound1}, Lemma \ref{ordroots}, Remark \ref{nested} and Proposition \ref{propgeneric}.\end{proof}

\begin{example}\label{exe101}
Consider the germ of curve $(X,0) \subset (\mathbb{C}^{200},0)$ parametrized by the following map:

\begin{center}
$\varphi(u):=(u^{2017}, \ u^{2018}, \ u^{2019},\cdots, \ u^{2216})$.
\end{center}

\noindent Note $2017$ is a prime number (unfortunately $2021$ is not). So the order of all elements $\theta \in G_{2017}\setminus \lbrace 1 \rbrace$ is $2017$. Hence, by Lemma \rm\ref{ordroots}, \textit{it is enough to test only one $\theta$ in $C_5$-procedure. In this case, $C_5(X,0)$ is the plane $x_1x_2$, that is, $C_5(X,0)=V(x_3,x_4,\cdots, x_{200})$, where $(x_1,x_2,\cdots,x_{200})$ are the local coordinates of $\mathbb{C}^{200}$. Note that without the use of Lemma} \rm\ref{bound2}, \textit{following Theorem} \rm\ref{generalcase1} \textit{we would need to use all the $2016$-roots of unity distinct from $1$.}
\end{example}

The previous example motivates the following corollary.

\begin{corollary}
If $(X,0)$ is irreducible and $m(X,0)$ is a prime number, then $C_5(X,0)$ is irreducible.
\end{corollary}

  To calculate the bound of proposition \ref{bound2} all you need to know are the multiplicities of the branches. However, for an irreducible curve $(X,0)$ 
  you can sharpen the bound if you know the set of characteristic auxiliary multiplicities $ChAM(X,0)$.
  
\begin{proposition}\label{bound3} \textbf{(Upper bound $3$)} Let $(X,0)=(X^1 \cup \cdots \cup X^r,0)$ be a germ of singular curve in $(\mathbb{C}^n,0)$. Then

\begin{center}
 $\sharp \lbrace Irred(C_5(X,0)) \rbrace \leq  \left(  \displaystyle { \sum_{i\in S(X)}^{}} \sharp \lbrace ChAM(X^{(i)},0) \rbrace -1 \right) +  \left( \displaystyle { \sum_{(i,j)\in T(X)}^{}} 
  \textrm{gcd}\left(m^{(i)}, m^{(j)} \right) \right) + \sharp \lbrace NT(X) \rbrace$.
\end{center}

\noindent In particular, if $(X,0)$ is irreducible, then $\sharp \lbrace Irred(C_5(X,0)) \rbrace \leq \sharp \lbrace ChAM(X,0) \rbrace -1 $.
\end{proposition}

   To prove this proposition it is convenient to recall the Lipschitz 
saturation $(X^S,0)$ of a branch. Let $O_{X,0}$ be the analytic 
algebra associated 
to the branch $(X,0)$. Algebraically, the Lipschitz saturation $O^S_{(X,0)}$
is the ring of meromorphic functions on $(X,0)$ which are locally 
Lipschitz with respect to the ambient metric; it is contained in 
the normalization $\overline{O_{X,0}} \cong \C\{u\}$.  The inclusion map:
   \[O_{X,0} \hookrightarrow O^S_{X,0}\]
has a geometric counterpart:   
\[ \zeta:(X^S,0) \to (X,0)\]
that can be realized as the restriction to $X^S$ of a $C_5$-generic 
linear projection, in the sense that its kernel is transversal to $C_5(X^S,0)$,
 see \cite[Prop. 8.5.20 \& Cor. 8.5.22]{GiSiTe}. The map $\zeta$ is then a 
bi-Lipschitz homeomorphism. Moreover, the analytic curve $(X^S,0)$ is 
a monomial curve that (determines and) is
   completely determined by $ChAM(X,0)=\{m,\beta_1,\ldots, \beta_g\}$.  It is isomorphic to the monomial curve with analytic algebra {\small $$O_{X,0}^S\cong \C\{u^m,u^{2m},\ldots ,u^{\beta_1},u^{\beta_1+e_1},\ldots,u^{\beta_2},u^{\beta_2+e_2},\ldots ,u^{\beta_3},\ldots ,u^{\beta_g},u^{\beta_g+1},\ldots\},$$}

   A fairly detailed account of all of this can be found in  \cite[Sec. 8.5.2]{GiSiTe}.  

 \begin{example}
    Let $(X,0) \subset (\C^2,0)$ be the plane branch with normalization map:
    \begin{align*}
         \eta:(\C,0) &\longrightarrow (X,0) \\
               u & \mapsto (u^4,u^6+u^7)
   \end{align*}
   Then $ChAM(X,0)=\{4,6,7\}$ and 
   \[O_{X,0}^S\cong \C\{ u^4,u^6,u^7,u^8,u^9,\cdots\} \]
   
  and so we have the normalization map for the Lipschitz saturation $(X^S,0) \subset (\C^4,0)$ given by:
   \begin{align*}
         \eta^s:(\C,0) &\longrightarrow (X^s,0) \\
               u & \mapsto (u^4,u^6,u^7,u^9).
   \end{align*}
   By making the change of coordinates in $(\C^4,0)$, $(x,y,z,w) \mapsto (x,y+z,z,w)$ we can view the Lipschitz saturation map 
   \[ \zeta: (X^s,0) \to (X,0)\]
   as the projection on the first two coordinates. Note that
   \[C_5(X^S,0)=V(z_3,z_4) \cup V(z_2,z_4).\]

 \end{example}
 
 \begin{proof} (Of proposition \ref{bound3}) \\
     Let $(X,0) \subset (\C^n,0)$ be an irreducible curve with $ChAM(X,0)=\{m,\beta_1,\ldots,\beta_g\}$. Since 
     {\small $$O_{X,0}^S\cong \C\{u^m,u^{2m},\ldots ,u^{\beta_1},u^{\beta_1+e_1},\ldots,u^{\beta_2},u^{\beta_2+e_2},\ldots ,u^{\beta_3},\ldots ,u^{\beta_g},u^{\beta_g+1},\ldots\},$$}
     we have that the normalization map for the Lipschitz saturation is of the form: 
     \begin{align*}
         \eta^s:(\C,0) &\longrightarrow (X^s,0)\subset (\C^N,0) \\
               u & \mapsto (u^m,u^{\beta_1}, \ldots, u^{\beta_g},u^{j_1}, \ldots, u^{j_s}) 
   \end{align*}
    where $j_i$ is of the form $\beta_k + re_k$ for some $k \in \{1,\ldots,g\}$. This implies that
    \[C_5(X^S,0) = V(z_3,\ldots,z_N) \cup V(z_2,z_4,\ldots,z_N) \cup \cdots \cup V(z_2,\ldots,z_{g-1},z_{g+1},\ldots, z_N) \]
   and so $$\sharp \lbrace Irred(C_5(X^S,0)) \rbrace = \sharp \lbrace ChAM(X,0) \rbrace -1. $$
   But since the saturation map  $\zeta:(X^S,0) \to (X,0)$ is a  bi-Lipschitz homeomorphism and a $C_5$-generic projection it induces 
   a surjective map
   \[ \zeta: C_5(X^S,0) \to C_5(X,0)\] 
   sending each irreducible component of $C_5(X^S,0)$ surjectively onto an irreducible component of $C_5(X,0)$. In particular
   $$\sharp \lbrace Irred(C_5(X,0)) \rbrace \leq \sharp \lbrace ChAM(X,0) \rbrace -1, $$ 
which implies the required inequality in Proposition \ref{bound3}
 \end{proof}

\section{The number of planes of $C_5(X,0)$ is not a bi-Lipschitz invariant}
\label{number}

$ \ \ \ \ $  In this section we would like to take a quick look at the behavior of the $C_5$-cone in a family of equisingular curves. We refer to \cite[Sec. 2]{otoniel6} for most of the concepts on equisingularity mentioned in this section.

Consider a flat family of reduced curves $p:(\mathfrak{X},0)\rightarrow (\mathbb{C},0)$ and let $p:\mathfrak{X}\rightarrow T$ be a good representative (see \cite[p. $248$]{buch}, see also \cite[Def. 2.2]{otoniel5}). We denote the fibers of $p$ by $X_t:=p^{-1}(t)$, $t \in T$.
Recall that when $p:\mathfrak{X}\rightarrow T$ is bi-Lipschitz equisingular, then for each $t$, there exist a bi-Lipschitz homeomorphism from $X_0$ to $X_t$. It has been proved that in bi-Lipschitz equisingular families of curves, the number of irreducible components of the $C_3$-cone is constant, see \cite{otoniel6} and also \cite{edson} for a more general situation. This implies that $C_3(X_0,0)$ and $C_3(X_t,0)$ are homeomorphic. It is thus natural to address the following:

\begin{flushleft}
\textbf{Question:} If $p:\mathfrak{X}\rightarrow T$ is a bi-Lipschitz equisingular family of reduced curves, are the cones $C_5(X_t,0)$, $t\neq 0$, and $C_5(X_0,0)$ homeomorphic?
\end{flushleft}

Let us first remark that this question can also be formulated for a pair of curves. Consider for instance the curve $(X,0)$ of Example \ref{exe3}. By Theorem \ref{bilipequivalence} and Proposition \ref{propgeneric} $(X,0)$ and its $C_5$-generic projection $(\tilde{X},0)$ are bi-Lipschitz equivalent. The cones $C_5(X,0)$ and $C_5(\tilde{X},0)$ are not homeomorphic, since $C_5(X,0)$ is composed of two (distinct) planes, while $C_5(\tilde{X},0)$ is the plane that contains $(\tilde{X},0)$. 

We will see in the following example that the answer to this question is negative even for families of reduced curves. 

\begin{example}\label{exeplanos} Consider the germ of reduced analytic surface $(\mathfrak{X},0)$ in $(\mathbb{C}^4,0)$ given as the image of the map germ $\varphi:(\mathbb{C}^2,0)\rightarrow (\mathfrak{X},0)$, $\varphi(u,t)=(\varphi_t(u),t)$, where $\varphi_t(u)$ is defined as

\begin{equation}\label{eq7}
\varphi_t(u)=(u^{6},u^{9}+u^{10},u^{11}+tu^{10}).
\end{equation}

Consider the canonical projection $p:(\mathbb{C}^4,0)\rightarrow (\mathbb{C},0)$ to the last factor, where $x,y,z,t$ are local system of coordinates in $\mathbb{C}^4$. Using {\sc Singular} \rm\cite{singular}, \textit{one can find the reduced structure of $(\mathfrak{X},0)$ and check that it is a Cohen-Macaulay surface and its singular locus is the $t$-axis.}

\textit{Consider the restriction of the projection to a good representative $p:\mathfrak{X} \rightarrow T$. It is a bi-Lipschitz equisingular family of reduced curves. In fact,  note that $\varphi^{-1}((0,0,0,t))=(0,t)$ for all $t \in T$, since $\varphi$ is the normalization of $\mathfrak{X}$ then $X_t$ is irreducible for all $t\in T$. Since $\mathfrak{X}$ is Cohen-Macaulay, then $X_t$ is reduced for all $t$. Therefore, $\varphi_t$ is a Puiseux parametrization of $X_t$. Note that $ChAM(X_t,(0,0,0,t))=\lbrace 6,9,10 \rbrace$ for all $t \in T$, thus by Proposition} \rm\ref{propgeneric} \textit{the family of $C_5$-generic plane projections $\tilde{X}_t$ is a topologically trivial family of plane curves, the bi-Lipschitz equisingularity of $p:\mathfrak{X}\rightarrow T$ follows from} \rm\cite[Cor. 3.6]{otoniel6}.

\textit{Now note that for $t \neq 0$ $C_5(X_t,(0,0,0,t))$ is composed of two (distinct) planes, while $C_5(X_0,0)$ is only a plane. We conclude that the number of irreducible components of the $C_5$-cone is not a bi-Lipschitz invariant, not even in family.}   
\end{example}

We finish with a remark on the analytic equivalence between curves.

\begin{remark}

The $C_5$-cone of an arbitrary analytic set is invariant under biholomorphic transformations (see \rm\cite{chirka}, \textit{p.} \rm 92). \textit{In the case of curves, the number of planes of $C_5$ is then an analytic invariant. In this sense, we can use this criterion to construct curves which are bi-Lipschitz equivalent but are not analytic equivalent. For instance, Table }\rm\ref{tabela2} \textit{shows four curves which are two by two bi-Lipschitz equivalent, each of them with a different number of planes in its $C_5$-cone. Then, these four curves have distinct analytic types.}
\end{remark}

\begin{table}[!h]
\caption{Distinct analytic types of bi-Lipschitz equivalent germs of curves in $(\mathbb{C}^3,0)$}\label{tabela2}
\centering 
{\def\arraystretch{2}\tabcolsep=10pt
\begin{tabular}{@{}l | l | c | l @{}}

\hline			
Name &  Parametrization of $(X^{(i)},0)$ & $\sharp \lbrace Irred(C_5(X^{(i)},0)) \rbrace$ & Equation of $C_5(X^{(i)},0)$ \\
  
\hline  
$(X^{(1)},0)$ &  $(u^{16},u^{57},u^{24}+u^{36}+u^{54}+u^{55})$ & $1$ & $V(y)$ \\

$(X^{(2)},0)$ &  $(u^{16},u^{24}+u^{57},u^{36}+u^{54}+u^{55})$ & $2$ & $V(yz)$ \\

$(X^{(3)},0)$ &  $(u^{16},u^{24}+u^{36}+u^{57},u^{36}+u^{54}+u^{55})$ & $3$ & $V(yz(y-z))$ \\

$(X^{(4)},0)$ &    $(u^{16},u^{24}+u^{36}-u^{54}+u^{57},u^{36}+u^{54}+u^{55})$ & $4$ & $V(yz(y-z)(y+z))$ \\
  \hline 
\end{tabular}
}
\end{table}

\begin{flushleft}
\textit{Acknowledgements:} The authors would like to thank Tadeusz 
Krasi\'nski for showing them his paper \cite{Krasinsky}. A. Giles Flores acknowledges support by Conacyt grant 221635. O.N. Silva acknowledges support by São Paulo Research Foundation (FAPESP), grant 2020/10888-2. J. Snoussi acknowledges support by Conacyt grant 282937.
\end{flushleft}


\normalsize

\begin{center}

\end{center}

\begin{flushleft}
$\bullet$ Giles Flores, A.\\
\textit{arturo.giles@cimat.mx}\\
Universidad Aut\'onoma de Aguascalientes, Departamento de Matemáticas y Física, Aguascalientes, México.\\

$ \ \ $\\

$\bullet$ Silva, O.N.\\
\textit{otoniel@dm.ufscar.br}\\
Universidade Federal de São Carlos, Caixa Postal 676, 13560-905 São Carlos, SP, Brazil.\\

$ \ \ $\\

$\bullet$ Snoussi, J.\\
\textit{jsnoussi@im.unam.mx}\\
Universidad Nacional Aut\'onoma de M\'exico (UNAM), 
Instituto de Matem\'aticas, Unidad Cuernavaca, M\'exico.\\
\end{flushleft}


\begin{thebibliography}{9}

\bibitem{briancon} Brian\c{c}on, J., Galligo, A., Granger, M.: \textit{D\'eformations \'equisinguli\`eres des germes de courbes gauches r\'eduites}, M\'em. Soc. Math. France (N.S.) 69, no. 1 (1980/81).

\bibitem{buch} Buchweitz, R.O.; Greuel, G.-M.: \textit{The Milnor Number and Deformations of Complex Curve Singularities}. Invent. Math. \textbf{58}, 241–-281 (1980).

\bibitem{chirka} Chirka, E.M.: \textit{Complex Analytic sets}. Translated from the Russian by R. A. M. Hoksbergen. Mathematics and its Applications (Soviet Series), 46. Kluwer Academic Publishers Group, Dordrecht, (1989). 

\bibitem{otoniel6} Giles-Flores, A.; Silva, O.N.; Snoussi, J.: \textit{On Tangency in equisingular families of curves and surfaces}, \textit{Quart. J. Math. Oxford} \textbf{71}(2), 485--505 (2020).

\bibitem{GiSiTe} Giles Flores, A.; Silva,  O. N.; Teissier, B.  \textit{The biLipschitz Geometry of Complex Curves: an Algebraic approach}. Lecture Notes in 
Mathematics {\bf 2280} (2020), 217--272.

\bibitem{greuel6} Greuel, G.-M., Lossen, C., Shustin, E.: \textit{Introduction to singularities and deformations}, Springer Monographs in Mathematics. Springer, Berlin, (2007).


\bibitem{singular} Greuel, G.-M., Pfister, G., Sch\"{o}nemann, H., Wolfram, D.: {\sc Singular} 4-0-2, \textit{A computer algebra system for polynomial computations}, http://www.singular.uni-kl.de, (2015).

\bibitem{Krasinsky} Krasi\'nsky, T. : \textit{The join of algebraic curves}, Illinois Journal of Mathematics \textbf{46 (3)} (2002), 723-738.

\bibitem{pichon} Neumann, W.D., Pichon, A.: \textit{Lipschitz geometry of complex curves}. J. Singul. 10 (2014).

\bibitem{pham} Pham, F., Teissier, B.: \textit{Saturation des alg\`ebres analytiques locales de dimension un}, appendice in: Brian\c{c}on, J.; Galligo, A.; Granger, M. D\'eformations \'equisinguli\`eres des germes de courbes gauches r\'eduites. M\'em. Soc. Math. France, deuxi\`eme s\'erie, tome 1 No. 1 (1980) , 1-69.


\bibitem{edson} Sampaio, J.E.: \textit{Bi-Lipschitz homeomorphic subanalytic sets have bi-Lipschitz homeomorphic tangent
cones}, Selecta Math. (N.S.), \textbf{22}, 553--559 (2016).

\bibitem{otoniel5} Silva, O.N.; Snoussi, J.: \textit{Whitney equisingularity in one parameter families of generically reduced curves}, \textit{Manuscripta Math.} \textbf{163}, 463--479 (2020).

\bibitem{jsnoussi-LNM} Snoussi, J.: \textit{A quick trip into local singularities of complex curves and surfaces}, Lecture Notes in 
Mathematics {\bf 2280} (2020), 45--71.

\bibitem{stutz} Stutz, J.: \textit{Equisingularity and equisaturation in codimension $1$}. Amer. J. Math. 94 (1972), 1245-1268.

\bibitem{Teissier2} Teissier, B.: \textit{Vari\'et\'es polaires II: multiplicit\'es polaires, sections planes, et conditions de Whitney}, in ``Algebraic Geometry, Proceedings, La Rabida, 1981," Lecture Notes in Math. 961, Springer-Verlag, 1982, pp. 314-491.

\bibitem{whitney} Whitney, H.: \textit{Local properties of analytic varieties}. In: Diff. and Combinator. Topology. Princeton Univ. Press, (1965), 205-244.

\end{thebibliography}
\end{document}